\documentclass{article}
\usepackage{amssymb}
\usepackage{amsmath}
\usepackage{arxiv}
\usepackage{authblk}
\usepackage[utf8]{inputenc}
\usepackage[T1]{fontenc}
\usepackage{url}
\usepackage{booktabs}
\usepackage{amsfonts}
\usepackage{nicefrac}
\usepackage{microtype}
\usepackage{color}
\usepackage{lipsum}
\usepackage{amsthm}
\usepackage{indentfirst}
\usepackage{graphicx}
\usepackage{epstopdf}
\usepackage{fourier}
\usepackage{bm}
\usepackage{mathtools}
\usepackage{latexsym,enumerate}
\usepackage{multirow}
\usepackage{multicol}
\usepackage[skip=2pt]{caption}
\usepackage[font=small,skip=0pt]{subcaption}
\usepackage{array}
\usepackage{ascii}
\usepackage{algorithm,algorithmic}
\usepackage{mathtools}
\usepackage{arydshln}
\DeclareMathAlphabet{\mathcal}{OMS}{cmsy}{m}{n} % use original mathcal
\usepackage{hyperref}
\newcommand{\comment}[1]{}

\newcommand{\BEA}{\begin{eqnarray}}
\newcommand{\EEA}{\end{eqnarray}}

\newcommand{\BR}{\mathbb{R}}

\newtheorem{thm}{Theorem}[section]
\newtheorem{prop}[thm]{Proposition}

\newtheorem{lem}[thm]{Lemma}

\newcommand{\PreserveBackslash}[1]{\let\temp=\\#1\let\\=\temp}
\newcolumntype{C}[1]{>{\PreserveBackslash\centering}p{#1}}
\newcolumntype{R}[1]{>{\PreserveBackslash\raggedleft}p{#1}}
\newcolumntype{L}[1]{>{\PreserveBackslash\raggedright}p{#1}}

\newcommand{\stkout}[1]{\ifmmode\text{\sout{\ensuremath{#1}}}\else\sout{#1}\fi}

\begin{document}

\title{Solving forward and inverse PDE problems on unknown manifolds via physics-informed neural operators}
% \author{Anran, Yan, Harlim, Lu 
% }
\author[1]{Anran Jiao}
\author[2]{Qile Yan}
\author[3]{John Harlim}
\author[1,4,*]{Lu Lu}

\affil[1]{Department of Statistics and Data Science, Yale University, New Haven, CT 06511, USA}
\affil[2]{School of Mathematics, University of Minnesota, 206 Church St SE, Minneapolis, MN 55455, USA}
\affil[3]{Department of Mathematics, Department of Meteorology and Atmospheric
Science \\
Institute for Computational and Data Sciences \\
The Pennsylvania State University, University Park, PA 16802, USA}
\affil[4]{Wu Tsai Institute, Yale University, New Haven, CT 06510, USA}

\affil[*]{Corresponding author. Email: lu.lu@yale.edu}

% \date{}
\maketitle

\begin{abstract}
In this paper, we evaluate the effectiveness of deep operator networks (DeepONets) in solving both forward and inverse problems of partial differential equations (PDEs) on unknown manifolds. By unknown manifolds, we identify the manifold by a set of randomly sampled data point clouds that are assumed to lie on or close to the manifold. When the loss function incorporates the physics, resulting in the so-called physics-informed DeepONets (PI-DeepONets), we approximate the differentiation terms in the PDE by an appropriate operator approximation scheme. For the second-order elliptic PDE with a nontrivial diffusion coefficient, we approximate the differentiation term with one of these methods: the Diffusion Maps (DM), the Radial Basis Functions (RBF), and the Generalized Moving Least Squares (GMLS) methods. For the GMLS approximation, which is more flexible for problems with boundary conditions, we derive the theoretical error bound induced by the approximate differentiation. Numerically, we found that DeepONet is accurate for various types of diffusion coefficients, including linear, exponential, piecewise linear, and quadratic functions, for linear and semi-linear PDEs with/without boundaries. When the number of observations is small, PI-DeepONet trained with sufficiently large samples of PDE constraints produces more accurate approximations than DeepONet. For the inverse problem, we incorporate PI-DeepONet in a Bayesian Markov Chain Monte Carlo (MCMC) framework to estimate the diffusion coefficient from noisy solutions of the PDEs measured at a finite number of point cloud data. Numerically, we found that PI-DeepONet provides accurate approximations comparable to those obtained by a more expensive method that directly solves the PDE on the proposed diffusion coefficient in each MCMC iteration.
\end{abstract}


\lhead{}

\section{Introduction}

Solving partial differential equations (PDEs) on manifolds is crucial across various fields including the natural sciences and practical engineering. For example, in image processing, PDEs on surfaces have been used in image segmentation \cite{tian2009segmentation}, image inpainting \cite{shi2017weighted}, and restoration of damaged patterns \cite{bertalmio2001navier,macdonald2010implicit}. In computer graphics, applications include flow field visualization \cite{bertalmio2001variational}, surface reconstruction \cite{zhao2001fast}, and brain imaging \cite{memoli2004implicit}. In physics, such a problem arises in granular flow \cite{rauter2018finite} and phase ordering \cite{schoenborn1999kinetics} on surfaces, and liquid crystals on deformable surfaces \cite{nitschke2020liquid}. With such diverse applications, many numerical methods have been developed to solve PDEs on manifolds. For example, the surface finite element method (FEM) \cite{dziuk2007surface,dziuk2013finite} is a robust and efficient method when a triangular mesh is given on surface. However, when the manifold is identified by randomly sampled data point clouds, the triangular mesh can be difficult to obtain. To address such an issue, mesh-free approaches were developed. For example, several collocation methods have been developed, including the global Radial Basis Function methods \cite{piret2012orthogonal,fuselier2013high} and the RBF-generated finite difference (FD) methods \cite{shankar2015radial,lehto2017radial}. In these approaches, they first employ a manifold parametrization scheme, such as local SVD, level set methods \cite{bertalmio2001variational,greer2006improvement,xu2003eulerian}, closest point methods \cite{ruuth2008simple,petras2018rbf}, and orthogonal gradient methods \cite{piret2012orthogonal}, and subsequently approximate the surface differentiation along the approximate tangent bundle. Another class of mesh-free approach is to identify a regression solution to the PDE by employing the Generalized Moving Least-Squares (GMLS) \cite{liang2013solving,suchde2019meshfree,gross2020meshfree,jiang2024generalized} to approximate tangential derivatives on the point cloud data. Alternatively, graph-based approaches, including Graph Laplacian, Diffusion Maps, and Weighted Nonlocal Laplacian \cite{li2016convergent,li2017point,gh2019,jiang2020ghost,yan2022ghost}, do not require a parameterization of a point cloud and can handle randomly sampled data on high-dimensional manifolds, although limited to a certain class of differential operators.

While many mesh-free solvers are available on manifolds as listed above, the size of the resulting discrete approximation of the differential operator increases as a function of the size of the point cloud data. Beyond the storage issue, it can become a computational bottleneck when applying these solvers in inverse problem algorithms for parameter estimation in PDEs \cite{harlim2020kernel,harlim2022graph}. Whether using a maximum likelihood or Bayesian approach in the inversion method, the numerical algorithm typically requires an iterative procedure. In this procedure, the predicted observations corresponding to the proposed parameter value at the current iteration are compared to the measured observations. When the solution operator that maps the parameter to be determined to the PDE solution is not explicitly available, the computational cost of each iteration is dominated by the complexity of the PDE solver. 

% In particular, we consider the following elliptic PDE,
% \begin{equation}
% -\text{div}_g(\kappa \textup{grad}_g u)+c u=f\quad\text{on}\ M.
% \label{eq1}
% \end{equation}
% Here, $M$ is a $d$-dimensional embedded compact submanifold of Euclidean space $\mathbb{R}^n$ without boundary, the differential operators are defined with respect to the Riemannian metric $g$, the parameter $c:M\to \BR$ is a positive function, $\kappa:M\to \BR$ is a positive diffusion coefficient, and $f$ is a real-valued function defined on $M$. If the corresponding inverse problem is to approximate $\kappa$, then we need an efficient solver to evaluate the following solution operator $\kappa\mapsto u(\kappa,\cdot)$,
% where $u(\kappa,\cdot):M \to \mathbb{R}$ satisfies the PDE in \eqref{eq1}.

Recently, deep learning techniques have been utilized for solving PDEs as a more efficient solver in the field of scientific machine learning (SciML)~\cite{karniadakis2021physics}. Physics-informed neural networks (PINNs) have been developed by embedding the PDE loss into the loss function of neural networks~\cite{raissi2019pinn,lu2021deepxde,lu2021physics,pang2019fpinns,zhang2019quantifying,yu2022gradient,wang2023learning} and successfully applied for solving forward as well as inverse problems of PDEs across various fields \cite{chen2020physics,yazdani2020systems,daneker2023systems,wu2023effective,wu2024identifying,hao2023pinnacle,fan2023deep,daneker2024transfer}. In this spirit, an algorithm based on neural networks along with Diffusion Maps was proposed to solve elliptic PDEs on unknown manifolds with constraints on the PDE and boundary conditions \cite{liang2024solving}. Despite the advancements, this class of approaches is not numerically suitable for solving parameter estimation problems since it requires one to retrain the neural network model to find the PDE solution subjected to the new proposed parameter at each iteration. Alternatively, recent developments of deep neural operators to learn PDE solution operators by using neural networks overcome this limitation, such as deep operator network (DeepONet) \cite{lu2021learning, wang2021learning, jin2022mionet, lu2022multifidelity,zhu2023reliable,moya2024conformalized,jiang2023fourier}, Fourier neural operator \cite{li2020fourier, lu2022comprehensive, zhu2023fourier}, and other neural operator learning methods \cite{cai2021deepm,mao2021deepm,tripura2023wavelet, jiao2021one, liu2024multi, cao2023residual,yin2024dimon, zhang2024d2no}. 
This approach accelerates the online prediction of PDE solutions under varying conditions, such as different coefficients or boundary conditions~\cite{lin2021operator,di2023neural,mao2023ppdonet}, by an offline training of deep neural operators.

In this paper, we propose to learn a solution operator of elliptic PDEs on unknown manifolds using the DeepONet and physics-informed DeepONet (PI-DeepONet). Furthermore, we integrate the learned PI-DeepONet into a Bayesian Markov Chain Monte Carlo (MCMC) framework for solving inverse problems. Specifically, our contributions are summarized below.
\begin{enumerate}
    \item We employ DeepONets to learn the solution operator of an elliptic PDE on unknown manifolds from varying diffusion coefficients to PDE solutions. The effectiveness of our methods is demonstrated on different types of diffusion coefficients, linear/semi-linear PDEs, and torus/semi-torus manifolds with/without boundary conditions.
    \item We develop PI-DeepONets to further improve the performance, in which we incorporate physics and approximate the differentiation term with one of these methods: the Diffusion Maps (DM), the Radial Basis Functions (RBF), and the Generalized Moving Least Squares (GMLS) methods.
    \item We discuss the error induced by numerical approximations of the PDE solution and differential operators on unknown manifolds.
    \item We integrate PI-DeepONet into a Bayesian MCMC framework to infer the diffusion coefficient from noisy solutions of the PDEs for comparable accuracy and significant reduction of computational costs.
\end{enumerate}

Our paper is organized as follows. In Section~\ref{sec:method}, we describe the PDE problem setup and introduce DeepONet and PI-DeepONet. Then we introduce the error induced by the approximate loss function and derive the error bounds in Section~\ref{sec:error}. In Section~\ref{sec:experiments}, we present the numerical experiments for solving forward problems. Finally, we present the Bayesian approach to inverse problems and the numerical performance in Section~\ref{sec:inverse_torus}.

\section{Methodology}
\label{sec:method}

%{\color{blue}John's comment: This section should start with a statement of problem to find the solution operator corresponding to PDE problems. Discuss the Deep-O-Net and PI-Deep-O-Net. While the section deep-o-net, etc, seems to be written disconnectedly from the methodology, I believe they should be lump into the same section. 

%Then we discuss the new things that are being considered in the paper (unknown domain in the sense that it only identifies by a finite number of randomly sampled data, and the needs to approximate differential operator on this unknown domain). 
%The following 3 subsections that discuss specific approximation of gradient operator should be in the appendix.}

%\subsection{Problem setting}
To simplify the discussion, we consider finding the solution operator, $ \mathcal{G}: \kappa \mapsto u$, corresponding to a PDE problem,
\begin{equation}
\begin{aligned}
-\text{div}_g(\kappa \textup{grad}_g u)+c u&=f\quad\text{on}\ M,\\\label{eq1}
\mathcal{B}(u) &= 0, \quad\text{on}\ \partial M.
\end{aligned}
\end{equation}
Here, $M$ is a $d$-dimensional compact submanifold of Euclidean space $\mathbb{R}^n$ with the boundary $\partial M$.  The differential operators are defined with respect to the Riemannian metric $g$, the parameter $c:M\to\BR$ is a positive function, $\kappa:M \to \BR$ is a positive diffusion coefficient, and $f$ is a known real-valued function defined on $M$. Here, $\mathcal{B}$ denotes the operator corresponding to the boundary conditions, e.g., Dirichlet and Neumann conditions.  

In our setup, we consider the manifold to be unknown in the sense that we have no access to either the Riemannian metric or the embedding function that parameterizes the manifold $M$. All we have is a randomly sampled point cloud data, $X=\{\mathbf{x}_i\}_{i=1}^N\subset M$. 
At the core of our approach, we need to approximate the differential operator from the dataset $X$ that is assumed to lie on the manifold $M$. We will consider several approaches, including the Diffusion Maps \cite{coifman2006diffusion} algorithm which can construct the specific differential operator in \eqref{eq1} as proposed in Refs.~\cite{gh2019,jiang2020ghost,yan2022ghost}, and more general operator approximators, including the Radial Basis Function (RBF) method \cite{fuselier2012scattered,fuselier2013high,harlim2023radial} and Generalized Moving Least-Squares (GMLS) method \cite{liang2013solving,suchde2019meshfree,gross2020meshfree,jiang2024generalized}, where the latter is more flexible to handle manifolds with boundaries. Our goal in considering these estimators is to demonstrate the robustness of DeepONet independent of the differential operator estimators. In Section SM1 of the supplementary material, we provide a short overview of these differential operator estimators. 

The remainder of this section will be organized as follows: In Section~\ref{sec:DeepONet}, we will give a short overview of DeepONet. In Section~\ref{sec:PIDeepONet}, we will give a short overview of PI-DeepONet. %In Section~\ref{sec:error}, we will discuss the error induced by numerical approximations of the  PDE solution and differential operators on unknown manifolds and give a brief overview of the existing error analysis for DeepONet.

\subsection{DeepONet}\label{sec:DeepONet}

We first review the deep operator network (DeepONet) and its architecture. DeepONet was proposed to learn nonlinear operator mappings between infinite-dimensional function spaces~\cite{lu2021learning} based on the universal operator approximation theorem \cite{chen1995universal}. 
Corresponding to the PDE problem in \eqref{eq1}, we define the solution operator
$\mathcal{G}(\kappa) = u(\cdot\,;\kappa)$. While the discussion throughout this paper focuses on $\kappa$ as input, one can extend the approximation to other input parameters, such as $c$ or other parameters in the boundary operators. In our context, DeepONet is a class of neural-network models that approximates $\mathcal{G}$. 

Particularly, let $\mathcal{G}_\theta$ denote the DeepONet approximation of  $\mathcal{G}$, where $\theta$ denotes the trainable parameters of the network. A DeepONet consists of two subnetworks. The branch net takes a discrete representation of $\kappa$, 
the vector,
$\kappa(\Xi) = \left\{\kappa(\boldsymbol{\xi}_1), \kappa(\boldsymbol{\xi}_2), \ldots, \kappa(\boldsymbol{\xi}_m)\right\},$
whose components consist of the input function evaluated at an arbitrary set of sensors $\Xi = \{\boldsymbol{\xi}_1, \boldsymbol{\xi}_2, \ldots, \boldsymbol{\xi}_m \}$
as the input, where $\boldsymbol{\xi}_i \in M$. Here, the sensor locations are not necessarily identical to any element in the training dataset $X$. The trunk net takes the location $\mathbf{x}\in M$ as an input. While we are interested in predicting the solution with $\kappa$ that is also evaluated at the input location of the trunk net, $\mathbf{x}\in M$, the fact that $\kappa$ can be represented by pointwise evaluations on arbitrary sensors gives more flexibility in the training as well as in the prediction when the actual $\kappa$ is only known up to its evaluation on the sensor locations. The DeepONet estimator evaluated at $\mathbf{x}\in M$ is an inner product of the branch and trunk outputs:
\begin{equation}
\mathcal{G}_\theta(\kappa(\Xi))(\mathbf{x}) = \sum_{k=1}^{p} \underbrace{b_k(\kappa(\boldsymbol{\xi}_1), \kappa(\boldsymbol{\xi}_2), \ldots, \kappa(\boldsymbol{\xi}_m))}_{\text{branch}} \underbrace{t_k(\mathbf{x})}_{\text{trunk}} + b_0, \label{DeepONetmap}
\end{equation}
where $b_0 \in \mathbb{R}$ is a bias, $\{b_1, b_2, \ldots, b_p\}$ are the $p$ outputs of the branch net, and $\{t_1, t_2,\ldots,t_p\}$ are the $p$ outputs of the trunk net (Fig.~\ref{fig:pideeponet}A).

To train a DeepONet, we generate labeled data consisting of inputs $\left\{ \kappa^{(k)},\Xi, (\mathbf{x}_i^{(k)})_{i = 1, \ldots N}\right\}_{k=1,\ldots, N_{OBS}}$ and outputs
$\{u(\mathbf{x}_i^{(k)};\kappa^{(k)})\}_{i=1,\ldots, N,k=1,\ldots, N_{OBS}}$.
Since $\kappa^{(k)}$ are known, they can be represented by the vector
$\kappa^{(k)}(\Xi)$, whose components consist of $\kappa^{(k)}$ evaluated at sensor locations, $\Xi=\{\boldsymbol{\xi}_1,\ldots, \boldsymbol{\xi}_m\} $.  The output training data, $\mathcal{G}(\kappa^{(k)})(\mathbf{x}_i^{(k)}):=u(\mathbf{x}_i^{(k)};\kappa^{(k)})$, are generated by the solution of the PDE. In such a case, the network parameters $\theta$ are obtained by minimizing the
empirical loss function defined based on the mean square error between the true $\mathcal{G}(\kappa^{(k)})$ and the network prediction $\mathcal{G}_\theta(\kappa^{(k)})$:
\begin{equation}
\mathcal{L}_{\text{OBS}}(\theta) = \frac{1}{N_{OBS}N} \sum_{k=1}^{N_{OBS}}\sum_{i=1}^{N} \left|\mathcal{G}_\theta(\kappa^{(k)}(\Xi))(\textbf{x}_i^{(k)}) - \mathcal{G}(\kappa^{(k)})(\textbf{x}_i^{(k)}) \right|^2,
\label{eq:ob_loss}
\end{equation}
where $X^{(k)}=\{\textbf{x}_1^{(k)},\ldots, \textbf{x}_N^{(k)}\}$ is the set of $N$ locations in the domain of $\mathcal{G}(\kappa^{(k)})$. In general (as in our setting), however, we do not have access to the analytical solution of the PDE. Instead, we only have the approximate solution 
obtained from either the Diffusion Maps, RBF, or GMLS, as mentioned before, which we denote as $\{\hat{u}(\mathbf{x}_i^{(k)};\kappa^{(k)})\}_{i=1,\ldots, N, k = 1,\ldots, N_{OBS}}$. Accounting for this practical issue, our training is performed on the approximate loss function, 
\begin{equation}
\tilde{\mathcal{L}}_{\text{OBS}}(\theta) = \frac{1}{N_{OBS}N} \sum_{k=1}^{N_{OBS}}\sum_{i=1}^{N} \left|\mathcal{G}_\theta(\kappa^{(k)}(\Xi))(\textbf{x}_i^{(k)}) - \hat{u}(\mathbf{x}_i^{(k)};\kappa^{(k)}) \right|^2.
\label{eq:approx_ob_loss}
\end{equation}
We will discuss the error induced by this approximation in Section~\ref{sec:error}.
%To demonstrate the effectiveness of DeepONet, we compare the DeepONet outputs with the true solutions on a testing data set that is independent to the training data set. 

\subsection{PI-DeepONet}\label{sec:PIDeepONet}

Combining physics-informed idea with the DeepONet, the physics-informed DeepONet (PI-DeepONet) was introduced in Ref.~\cite{wang2021learning}. 
For the general case, we define a parametric PDE $\mathcal{N}(\kappa, u) = 0$ subject to boundary conditions $\mathcal{B}(\kappa, u) = 0$ (Fig.~\ref{fig:pideeponet}B). The PDE problem in \eqref{eq1} is a concrete example, where 
$\mathcal{N}(\kappa, u) = -\text{div}_g(\kappa \textup{grad}_g u)+c u -f$ and
$\mathcal{B}$ is independent of $\kappa$.

\begin{figure}[htbp]
    \centering
    \includegraphics[width=12cm]{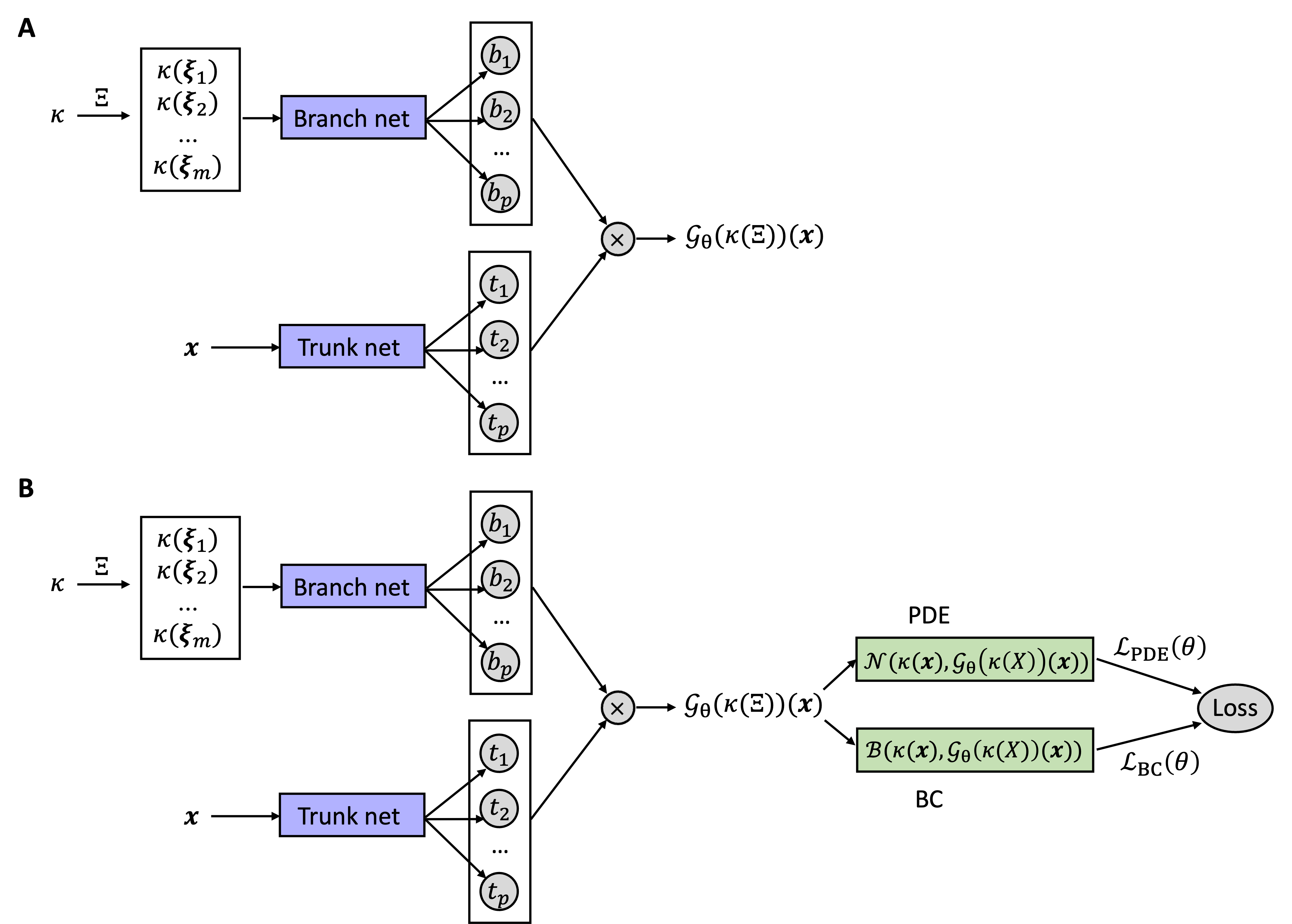}
    \caption{\textbf{The architecture of DeepONets.} (\textbf{A}) DeepONet. (\textbf{B}) PI-DeepONet.}
    \label{fig:pideeponet}
\end{figure}

%Other than the loss from observations $\mathcal{L}_{\text{OBS}}$ the same as (\ref{eq:ob_loss}), we add loss terms to enforce the underlying PDE constraints. 

The PI-DeepONet constraints the observation loss function $\mathcal{L}_{\text{OBS}}$ in (\ref{eq:ob_loss})
with an additional PDE residual loss term $\mathcal{L}_{\text{PDE}}$ and boundary conditions loss term $\mathcal{L}_{\text{BC}}(\theta)$, such that:
\begin{equation}
\mathcal{L}(\theta) = w_{\text{OBS}} \mathcal{L}_{\text{OBS}}(\theta) + 
 w_{\text{PDE}}\mathcal{L}_{\text{PDE}}(\theta) + w_{\text{BC}}\mathcal{L}_{\text{BC}}(\theta),
\label{eq:pi_loss}
\end{equation}
where $w_{\text{OBS}}$, $w_{\text{PDE}}$, and $w_{\text{BC}}$ are the weights for each term. In the original PI-DeepONet settings, the PDE loss term is given as, 
\begin{equation}
\mathcal{L}_{\text{PDE}}(\theta) 
= \frac{1}{N_{PDE}N} \sum_{k=1}^{N_{PDE}}\sum_{i=1}^{N} \left| \mathcal{N}\left(\kappa^{(k)}(\textbf{x}_i^{(k)}), \mathcal{G}_\theta\big(\kappa^{(k)}\big)(\textbf{x}_i^{(k)}) \right)\right|^2.\label{eq:pde_loss}
\end{equation}
When the differential operator is defined on the Euclidean domain, it can be handled 
by the automatic differentiation (AD). In this work, as we solve PDEs on the manifold, we cannot use AD for $\mathcal{L}_{\text{PDE}}$. Instead, the differential operator is approximated by Diffusion Maps, RBF, or GMLS and then incorporated into PDE loss, such that,
\begin{equation}
\mathcal{L}_{\text{PDE}}(\theta) 
\approx \tilde{\mathcal{L}}_{\text{PDE}}(\theta) =  \frac{1}{N_{PDE}N} \sum_{k=1}^{N_{PDE}}\sum_{i=1}^{N} \left| (\textbf{L} + c\textbf{I})\mathcal{G}_\theta\big(\kappa^{(k)}(X^{(k)})\big)(\textbf{x}_i^{(k)}) - f(\textbf{x}_i^{(k)}) \right|^2.
\label{eq:approx_pde_loss}
\end{equation}
We note that the numbers of inputs in the loss functions in \eqref{eq:ob_loss} and \eqref{eq:pde_loss}, respectively, $N_{OBS}$ and $N_{PDE}$,  are not necessarily equal. We denote the discrete approximation of $L^\kappa =-\text{div}_g(\kappa\textup{grad}_g \,)$ on $X^{(k)}=\left\{\mathbf{x}_1^{(k)},\ldots,\mathbf{x}_N^{(k)}\right\}$ by $\mathbf{L}$ that can be $\mathbf{L}^{DM}$, $\mathbf{L}^{RBF}$, or $\mathbf{L}^{GMLS}$ depending on the discretization method being used, which details are shown in Section SM1 of the supplementary material. 

When there are boundary conditions, we may have an additional loss term $\mathcal{L}_{\text{BC}}$ defined as:
\begin{equation}
\mathcal{L}_{\text{BC}}(\theta) 
= \frac{1}{N_{PDE}N} \sum_{k=1}^{N_{PDE}}\sum_{i=1}^{N_b} \left| \mathcal{B}\left(\kappa^{(k)}(\textbf{x}_i^{(k)}), \mathcal{G}_\theta\big(\kappa^{(k)}(X^{(k)})\big)(\textbf{x}_i^{(k)})\right) \right|^2, 
\label{eq:bcic_loss}
\end{equation}
where $\{\textbf{x}_i^{(k)}\in \partial M \}_{i=1}^{N_b}$ are $N_b$ locations on the boundary of $M$ if these points are available as in the Euclidean setting. Since the boundary is a measure zero set, we may not have (or be able to) data points exactly on the boundary since the manifold is unknown. To overcome this issue, we will impose the boundary conditions on points that are sufficiently close to the boundary. Several methods to detect such points are available in the literature. For examples, see \cite{berry2017density} for a method that estimates the distance of the sample points to the boundary and \cite{jiang2024generalized} for a method that does not estimate the distance to the boundary. In our numerical simulations, we will use the close-to-boundary detection method proposed in \cite{jiang2024generalized}. For example, if the PDE satisfies a nonhomogeneous boundary condition, $u=g$ on $\partial M$, then we define $\tilde{g}:U\subset M \to\BR$ as an extension of $g:\partial M \to\BR$ on a neighborhood of $\partial M$ in $M$, that is, $\tilde{g}|_{\partial M} = g$ (e.g., see p.380 of \cite{lee2018introduction} for the validity of such an extension). Define the set of points whose (geodesic) distance from the boundary is larger than $\epsilon>0$ as, 
\[
X_\epsilon^{(k)} = \left\{\mathbf{x}_i \in X^{(k)}: d_g(\mathbf{x}_i,\partial M) > \epsilon\right\},  
\]
and let $N_1(k) = |X_\epsilon^{(k)}|$ and $N_2(k) = |X^{(k)}\backslash X_\epsilon^{(k)}|$ such that $N = N_1(k)+ N_2(k)$, 
for every $k=1,\ldots, N_{PDE}$.
Then, the approximate PDE loss in \eqref{eq:approx_pde_loss} is averaged over $N_1(k)$ points for each $k$ in $X_\epsilon$ rather than over $N$ points in $X$, 
\begin{equation}
\tilde{\mathcal{L}}_{\text{PDE}}(\theta) =  \frac{1}{N_{PDE}} \sum_{k=1}^{N_{PDE}}\frac{1}{N_1(k)}\sum_{i=1}^{N_1(k)} \left| (\textbf{L} + c\textbf{I})\mathcal{G}_\theta\big(\kappa^{(k)}(X^{(k)})\big)(\textbf{x}_i^{(k)}) - f(\textbf{x}_i^{(k)}) \right|^2.
\label{eq:approx_pde_loss2}
\end{equation}
and the approximate boundary condition loss function is given as,
\begin{equation}
\tilde{\mathcal{L}}_{\text{BC}}(\theta) 
= \frac{1}{N_{PDE}} \sum_{k=1}^{N_{PDE}}\frac{1}{N_2(k)}\sum_{i=1}^{N_2(k)} \left|\mathcal{G}_\theta\big(\kappa^{(k)}(X^{(k)})\big)(\textbf{x}_i^{(k)}) - \tilde{g}(\textbf{x}_i^{(k)})  \right|^2,\label{eq:approx_bc_loss}
\end{equation}
where $\textbf{x}_i^{(k)} \in X^{(k)}\backslash X^{(k)}_\epsilon$.

\section{Error induced by the approximate loss function}\label{sec:error}

While the ideal training procedure is to minimize the empirical loss function $\mathcal{L}(\theta)$ defined in \eqref{eq:pi_loss}, this is practically inaccessible since the exact solution operator $\mathcal{G}(\kappa^{(k)})$ 
is intractable and the differential operator can only be approximated from the point cloud data when the manifold is unknown.
With this practical constraint, the training procedure is to minimize the following approximate loss function,
\begin{equation}
\tilde{\mathcal{L}}(\theta) = w_{\text{OBS}} \tilde{\mathcal{L}}_{\text{OBS}}(\theta) + w_{\text{PDE}}\tilde{\mathcal{L}}_{\text{PDE}}(\theta) + w_{\text{BC}}\tilde{\mathcal{L}}_{\text{BC}}(\theta),
\label{eq:approx_pi_loss}
\end{equation}
where $\tilde{\mathcal{L}}_{\text{OBS}}(\theta)$, $\tilde{\mathcal{L}}_{\text{PDE}}(\theta)$, and  $\tilde{\mathcal{L}}_{\text{BC}}(\theta)$
are given by \eqref{eq:approx_ob_loss},  \eqref{eq:approx_pde_loss2}, \eqref{eq:approx_bc_loss}, respectively, for our example. For manifolds with no boundary, one can set $w_{BC}=0$ and use the approximate loss function in \eqref{eq:approx_pde_loss}.

Let us focus the discussion below for the case with boundary condition. In this case, the training procedure induces an error,
% \begin{equation}
% \left|\mathcal{L}(\theta) - \tilde{\mathcal{L}}(\theta)\right| &= & |w_{OBS}| \left|\mathcal{L}_{OBS}(\theta) - \tilde{\mathcal{L}}_{OBS}(\theta)\right| + |w_{PDE}| \left|\mathcal{L}_{PDE}(\theta) - \tilde{\mathcal{L}}_{PDE}(\theta)\right|+ |w_{BC}| \left|\mathcal{L}_{BC}(\theta) - \tilde{\mathcal{L}}_{BC}(\theta)\right| \nonumber \\
% &\leq &  \frac{|w_{OBS}|}{N_{OBS}N} \sum_{k=1}^{N_{OBS}}\sum_{i=1}^{N} \left| \hat{u}(\mathbf{x}_i^{(k)};\kappa^{(k)}) - \mathcal{G}(\kappa^{(k)})(\textbf{x}_i^{(k)})  \right|^2  \notag \\ && + \frac{|w_{PDE}|}{N_{PDE}} \sum_{k=1}^{N_{PDE}}\frac{1}{N_1(k)}\sum_{i=1}^{N_1(k)} \left|\left(-\text{div}_g(\kappa\textup{grad}_g \,) - \textbf{L} \right) \mathcal{G}_\theta\big(\kappa^{(k)}(X^{(k)})\big)(\textbf{x}_i^{(k)}) \right|^2 \label{eq:error_loss} \\
% && + \frac{|w_{BC}|}{N_{PDE}} \sum_{k=1}^{N_{PDE}}\frac{1}{N_2(k)}\sum_{i=1}^{N_2(k)} \left|\mathcal{G}_\theta\big(\kappa^{(k)}(X^{(k)})\big)(\textbf{x}_i^{b(k)}) - g (\textbf{x}_i^{b(k)}) -\left( \mathcal{G}_\theta\big(\kappa^{(k)}(X^{(k)})\big)(\textbf{x}_i^{(k)}) - \tilde{g}(\textbf{x}_i^{(k)})\right)\right|^2,\notag
% \end{equation}
%\begingroup
%\makeatletter
%\def\f@size{8}
%\check@mathfonts
%\begin{subequations}
%\begin{align}
\BEA
  \left|\mathcal{L}(\theta) - \tilde{\mathcal{L}}(\theta)\right| & \leq & |w_{OBS}| \left|\mathcal{L}_{OBS}(\theta) - \tilde{\mathcal{L}}_{OBS}(\theta)\right| + |w_{PDE}| \left|\mathcal{L}_{PDE}(\theta) - \tilde{\mathcal{L}}_{PDE}(\theta)\right|  \notag \\ && +|w_{BC}| \left|\mathcal{L}_{BC}(\theta) - \tilde{\mathcal{L}}_{BC}(\theta)\right|, \label{eq:error_loss}
  \EEA
where the first term is,
\BEA
\left|\mathcal{L}_{OBS}(\theta) - \tilde{\mathcal{L}}_{OBS}(\theta)\right| \leq \frac{1}{N_{OBS}N}\sum_{k=1}^{N_{OBS}}\sum_{i=1}^{N} \left| \hat{u}(\mathbf{x}_i^{(k)};\kappa^{(k)}) - \mathcal{G}(\kappa^{(k)})(\textbf{x}_i^{(k)})  \right|^2, \notag 
\EEA
the second term is,
\BEA
\left|\mathcal{L}_{PDE}(\theta) - \tilde{\mathcal{L}}_{PDE}(\theta)\right| &\leq& \frac{1}{N_{PDE}}
\sum_{k=1}^{N_{PDE}}\frac{1}{N_1(k)} 
\notag \\ &&\sum_{i=1}^{N_1(k)} \left|\left(-\text{div}_g(\kappa\textup{grad}_g \,) - \textbf{L} \right) \mathcal{G}_\theta\big(\kappa^{(k)}(X^{(k)})\big)(\textbf{x}_i^{(k)}) \right|^2, \notag
\EEA
and the third term is,
\BEA
\left|\mathcal{L}_{BC}(\theta) - \tilde{\mathcal{L}}_{BC}(\theta)\right| \leq \frac{1}{N_{PDE}} \sum_{k=1}^{N_{PDE}}\frac{1}{N_2(k)} \hspace{5cm} \notag \\ \sum_{i=1}^{N_2(k)}\left|\mathcal{G}_\theta\big(\kappa^{(k)}(X^{(k)})\big)(\textbf{x}_i^{b(k)}) - g (\textbf{x}_i^{b(k)}) - \left( \mathcal{G}_\theta\big(\kappa^{(k)}(X^{(k)})\big)(\textbf{x}_i^{(k)}) - \tilde{g}(\textbf{x}_i^{(k)})\right)\right|^2,\notag
\EEA
\comment{
  \\
 & + \frac{|w_{PDE}|}{N_{PDE}} \sum_{k=1}^{N_{PDE}}\frac{1}{N_1(k)}\sum_{i=1}^{N_1(k)} \left|\left(-\text{div}_g(\kappa\textup{grad}_g \,) - \textbf{L} \right) \mathcal{G}_\theta\big(\kappa^{(k)}(X^{(k)})\big)(\textbf{x}_i^{(k)}) \right|^2 \label{eq:error_loss} \\
& + \frac{|w_{BC}|}{N_{PDE}} \sum_{k=1}^{N_{PDE}}\frac{1}{N_2(k)}\sum_{i=1}^{N_2(k)} \left|\mathcal{G}_\theta\big(\kappa^{(k)}(X^{(k)})\big)(\textbf{x}_i^{b(k)}) - g (\textbf{x}_i^{b(k)}) - \left( \mathcal{G}_\theta\big(\kappa^{(k)}(X^{(k)})\big)(\textbf{x}_i^{(k)}) - \tilde{g}(\textbf{x}_i^{(k)})\right)\right|^2,\notag
}
%\end{align}
%\end{subequations}
%\endgroup
where $x_i^{b(k)} \in \partial M$ denotes the closest point on the boundary to the data point  $x_i^{(k)}\in X^{(k)}\backslash X_{\epsilon}^{(k)}$. 

The three terms on the right-hand side above consist of the convergence of the PDE solver, the consistency of the numerical approximation of the differential operator, and the error induced by the approximate boundary conditions. The upper bounds for these three terms depend on the PDE solvers that are being used to generate the approximate solution, $\hat{u}(\cdot\,;\,\kappa^{(k)})$, and to approximate $\mathbf{L}$. For Diffusion Maps estimator, the error bound for the case without boundary conditions is of order-$N^{-\frac{4}{d+6}}$, which detailed analysis can be found in \cite{gh2019}. For the Diffusion Maps estimator with boundary (where the sample points on the boundary are also available), the error bounds are reported in \cite{jiang2020ghost} for various types of boundary conditions subject to appropriate regularity on the coefficients $\kappa$, $f$, and boundary operator $\mathcal{B}$. For the RBF approach, the consistency error bound in terms of randomly sampled data can be deduced following the same argument of the proof of Lemma~3.14 in \cite{yan2023spectral}. Since the RBF interpolator is known to be unstable (see Chapter 12 of \cite{Wendland2005Scat}), it is unclear how to bound the first term above, unless additional regularization is being used in attaining the RBF solution $\hat{u}$. 

For the GMLS estimator, we have the following error bounds.
\begin{lem}\label{lem21}
Let $X^{(k)} = \{\mathbf{x}_1^{(k)},\ldots,\mathbf{x}_N^{(k)}\}$ be a set of uniformly i.i.d. samples of $M$. Let $\kappa^{(k)} \in C^1(M)$ and the estimator, $\mathcal{G}_\theta\big(\kappa^{(k)}(X^{(k)})\big)\in C^{p+1}(M^*)$ function, where $M^* = \bigcup_{\mathbf{x}\in M} B(\mathbf{x},C_2N^{-1/d})$
is a union of geodesic ball with center $\mathbf{x}$ and radius $C_2N^{-1/d}$
and $p\geq 2$. Let $\mathbf{L}$ be the GMLS estimator of $-\textup{div}_g(\kappa^{(k)}\textup{grad}_g)$ with intrinsic polynomial of degree-$p$.
\begin{description}
    \item[(a)] Then, with probability higher than $1-\frac{1}{N}$,
    \[
    \left|\left(-\text{div}_g(\kappa\textup{grad}_g \,) - \textbf{L} \right) \mathcal{G}_\theta\big(\kappa^{(k)}(X^{(k)})\big)(\textbf{x}_i^{(k)}) \right| = O(N^{-\frac{p-1}{d}}).
    \]
\item[(b)] Let $\mathcal{G}(\kappa^{(k)}) \in C^{p+1}(M^*)$ be a classical solution of the PDE problem in \eqref{eq1} and $\hat{u}$ be the GMLS approximate solution obtained by the algorithm discussed in Section SM1.3 of the supplementary material, then with probability higher than $1-\frac{2}{N}$,
\[
\left| \hat{u}(\mathbf{x}_i^{(k)};\kappa^{(k)}) - \mathcal{G}(\kappa^{(k)})(\textbf{x}_i^{(k)})  \right| = O(N^{-\frac{p-1}{d}}).
\]
\end{description}
\end{lem}

The proof of this Lemma is identical to the discussion in Remark~3.1 and Theorem 4.1 of \cite{jiang2024generalized} which reported the error bounds for the case $\kappa^{(k)} = 1$.

With this lemma, we can immediately deduce the following result:

\begin{prop}
Let the assumption in Lemma~\ref{lem21} be satisfied. In addition, let $g:\partial M \to \BR$ be a Lipschitz function. Then, with probability higher than $1-\frac{4}{N}$,
\[
\left|\mathcal{L}(\theta) - \tilde{\mathcal{L}}(\theta)\right| \leq  C_1(|w_{OBS}|+|w_{PDE}|) N^{-\frac{2(p-1)}{d}}  + C_2 |w_{BC}| N^{-\frac{2}{d}},
\]
for some constants $C_1, C_2>0$ that are independent of $N$.
\end{prop}

\begin{proof}
The first term in the error bound above is immediately attained by inserting 
the upper bounds in Lemma~\ref{lem21} to the first two right-hand terms in
\eqref{eq:error_loss}. Since $g$ is Lipschitz, we can extend it on an $\epsilon$-neighbors of $\partial M$ that contains the set of points $X\backslash X_\epsilon$. Particularly, let $\tilde{g}$ be the Lipschitz function such that $\tilde{g}|_{\partial M} = g$. Define 
\[
h(\textbf{x}) := \mathcal{G}_\theta\big(\kappa^{(k)}(X^{(k)})\big)(\textbf{x}) - \tilde{g} (\textbf{x}). 
\]
 Since $\mathcal{G}_\theta(\kappa^{(k)}(X^{(k)})\big) \in C^{p+1}(M^*)$ for $p\geq 2$, then it is clear that $h$ is Lipschitz. Denoting $\mathbf{x}_i^{b(k)}$ as the point at the boundary whose geodesic distance is the closest to $\mathbf{x}_i^{(k)}$, we have: 
 \begin{equation}
|h(\mathbf{x}_i^{(k)}) - h(\mathbf{x}_i^{b(k)})  | \leq L d_g(\mathbf{x}_i^{(k)},\mathbf{x}_i^{b(k)}) \leq L h_{X,M},\label{eq:errorh}
 \end{equation}
 where $h_{X,M}= \sup_{\mathbf{x}\in M} \min_{j\in\{1,\ldots,N\}} d_g(\mathbf{x},\mathbf{x}_j)$ is the fill distance. Since the data are uniform i.i.d. samples, then with probability higher than $1-\frac{1}{N}$, $h_{X,M} \leq C(d) N^{-1/d}$ (see Lemma~B.2 in \cite{harlim2023radial} for the proof of this statement). Inserting \eqref{eq:errorh} to the last term in \eqref{eq:error_loss}, the proof is complete.
\end{proof}

Next, we discuss another error source: the approximation error of DeepONet. The theoretical foundation of DeepONet stems from the universal approximation theorem of shallow neural networks for operators in~\cite{chen1995universal}, which is further extended to deep neural networks \cite{lu2021learning,jin2022mionet}. Specifically, it is demonstrated that for any specified tolerance $\epsilon$, there exists a DeepONet such that the resulting approximation error is smaller than this tolerance.
% Also, it shows the exponential convergence of the DeepONet error with respect to the size of the training data \cite{lu2021learning}.
% DeepONet is then compare with Fourier neural operator theoretically and computationally, which can be regarded as a special form of branch network in DeepONet \cite{lu2022comprehensive}. 
% Furthermore, DeepONet is interpreted in terms of encoding, approximation and reconstruction errors and both lower and upper bounds can be estimated according to \cite{lanthaler2022error}.
Building on these results, various works have studied approximation error estimates of DeepONet for specific PDEs. For example, Ref.~\cite{lanthaler2022error} presented explicit error bounds of DeepONet in terms of its network size for elliptic PDEs. Marcati et al.~\cite{marcati2023exponential} provided an upper bound on the exponential expression rate of DeepONet for elliptic isotropic diffusion problems. Additionally, Deng et al.~\cite{deng2022approximation,lu2022comprehensive} developed theoretical approximation rates of learning solution operators from both linear and nonlinear advection-diffusion-reaction equations, and they found the approximation rates depend on the architecture of branch networks as well as the smoothness of inputs and outputs of solution operators. Moreover, it is proved that DeepONets can break the curse of dimensionality when approximating certain nonlinear operators \cite{lanthaler2022error, marcati2023exponential}.

% \subsection{A PDE loss term}
% We consider the residual of the PDE (\ref{eq1}) in $L^2$ norm as the PDE loss, i.e., 
% $$
% r(u;\kappa):=\left(\int_M(-\text{div}_g(\kappa \textup{grad}_g u)+c u-f)^2 d\text{Vol}\right)^{1/2}
% $$
% Given a set of independent random points $X=\{\mathbf{x}_i\}_{i=1}^N\subset M$ that is assumed to be uniformly distributed over $M$, then above residual $r(u,\kappa)$ could be discretized as: 
% $$
% \tilde{r}_\mu(u;\kappa):=\left(\frac{1}{N}\sum_{i=1}^N (-\text{div}_g(\kappa \textup{grad}_g u)+c u(\mathbf{x}_i)-f(\mathbf{x}_i))^2\right)^{1/2}.
% $$

% In practice, we use the RBF approach to approximate the differential operators $\text{div}_g$ and $\textup{grad}_g$ on manifolds. Then the loss term in they physics informed DeepONet is set to be:
% $$
% L_{physics}(\theta) :=\frac{1}{N_K}\sum_{k=1}^{N_K}\tilde{r}_\mu(\mathcal{G}_\theta(\kappa_k);\kappa_k).
% $$

\section{Numerical experiments}
\label{sec:experiments}

In this section, we numerically demonstrate the capability and effectiveness of our proposed DeepONet and PI-DeepONet for solving PDEs on manifolds. We consider solving forward problems including a linear PDE defined on the 2D torus (Section~\ref{sec:linear_torus}) and the 2D semi-torus (Section~\ref{sec:linear_semitorus}), a nonlinear PDE on the 2D torus (Section~\ref{sec:nonlinear_torus}).% In Section~\ref{sec:inverse_torus}, we solve an inverse problem on the 2D torus. 
%In each numerical experiment, we will specify (and compare results) when Diffusion maps, RBF, or GMLS are utilized to approximate surface differential operators.

To learn the solution operator of the PDE using DeepONet and PI-DeepONet, we first generate training dataset and test dataset as described in Section~\ref{sec:DeepONet} and Section~\ref{sec:PIDeepONet}. In this work, for a point $\mathbf{x}=(x, y, z)\in M$, we consider $\kappa$ as a function of $x$ and $y$, i.e., $\kappa=\kappa(x, y)$ for simplicity. One  can easily modify the network architecture to extend it to more general cases. For vanilla DeepONet, the training takes randomly specified $\kappa(\boldsymbol{\Xi})$ as the input for branch net, where the sensors $\Xi$ consists of equi-spaced (in intrinsic coordinates) sensors grid of $26 \times 26$ (i.e., $m=676$). When computing the loss $\mathcal{L}_{\text{OBS}}$, we use the $N_{OBS}$ discretized random functions $\kappa(\Xi)$ during training. 

For PI-DeepONet, similarly, the training takes $N_{PDE}$ and $N_{OBS}$ random $\kappa$ for $\mathcal{L}_{\text{PDE}}$ and $\mathcal{L}_{\text{OBS}}$, respectively. In both cases, we can directly employ classical networks or other network architecture designs for the branch net and trunk net. In our numerical simulations, we choose convolutional neural network (CNN) in the branch net, and feedforward neural network (FNN) with a depth of 3 layers and a width of 32 nodes for the trunk net. We use rectified linear units (ReLU, $\max(0,x)$)~\cite{nair2010rectified} and Gaussian error linear units (GELU, $x \cdot \frac{1}{2}\left[1+\text{erf} (x/\sqrt{2})\right]$)~\cite{hendrycks2016gaussian}) activation functions for branch net and trunk net, respectively. The initial learning rate is set to be $\gamma_0 = 0.001$, and we employ the inverse time learning rate decay technique with the formula
$\gamma_n = \frac{\gamma_0}{1 + rn/S},$
where $n$ is the number of iterations, $\gamma_n$ is the learning rate after $n$ iterations, $r$ is the decay rate, and $S$ is the decay step. In our experiments, we choose $r=0.5$ and $S=20000$ and use Adam optimizer.
% We show the hyperparameters of the neural networks in Table~\ref{tab:hyperparameters}.

% \begin{table}[htbp]
%     \caption{\textbf{The hyperparameters of vanilla DeepONet and PIDeepONet used in our experiments.}}
%     \label{tab:hyperparameters}
%     \centering
%     \begin{tabular}{ccc}
%     \toprule
%      Problems & Depth & Width\\
%     \midrule
%     Section~\ref{sec:linear_torus} Linear PDE on torus&  3 & 32 \\
%     Section~\ref{sec:linear_semitorus} Linear PDE on semi-torus& 8 & 16 \\
%     Section~\ref{sec:nonlinear_torus} Nonlinear PDE on torus& 8 & 16 \\
%     Section~\ref{sec:inverse_torus} Inverse problem & 16 & 32 \\
%     \bottomrule
% \end{tabular}
% \end{table}

To evaluate the performance of different models, we generate a test dataset that consists of inputs $\left\{ \kappa^{(j)},\Xi, (\mathbf{x}_i^{(j)})_{i = 1, \ldots N}\right\}_{j=1,\ldots, N_t}$ and outputs
$\{\mathcal{G}(\kappa^{(j)})(\mathbf{x}_i^{(j)}) = u(\mathbf{x}_i^{(j)};\kappa^{(j)})\}_{i=1,\ldots, N,j=1,\ldots, N_t}$. Here, $\left\{ \kappa^{(j)}\right\}_{j=1,\ldots, N_t}$ represents a distinct set of input functions differing from the training data, and the sensors $\Xi = \{\boldsymbol{\xi}_1, \boldsymbol{\xi}_2, \ldots, \boldsymbol{\xi}_m \}$ and locations $X=\{\mathbf{x}_i\}_{i=1}^N$ are the same as those used in the training dataset. We adopt the mean $L^2$ relative error defined as, 
\[
\mathcal{E}:= \frac{1}{N_t}\sum_{j=1}^{N_t}\frac{\parallel \mathcal{G}_\theta(\kappa^{(j)}(\Xi))(X^{(j)}) - \mathcal{G}(\kappa^{(j)})(X^{(j)}) \parallel_2}{\parallel \mathcal{G}(\kappa^{(j)})(X^{(j)}) \parallel_2}.
\]

% {\color{blue}This is strange. I thought you said when you apply training, $X^{(i)} = X$ for $i=1,\ldots, N_{PDE}$. In the testing, you should evaluate them on different data set $Z = \{\mathbf{x}_1, \ldots, \mathbf{x}_{N_t} \}$. I think the notation $\mathcal{G}_\theta(\kappa^{(i)})(X^{(i)})$ is not consistent with the previous discussion. What is the input sensor here?}

We run the test for three times and report the average value. The datasets are generated using MATLAB R2023a. For all experiments, the Python library DeepXDE \cite{lu2021deepxde} is utilized to implement the neural networks. All the codes and data will be available on GitHub at \href{https://github.com/lu-group/manifold-deeponet}{https://github.com/lu-group/manifold-deeponet}.

\subsection{Second-order linear elliptic PDE on torus}
\label{sec:linear_torus}

We first consider solving Eq.~(\ref{eq1}) with $c=1$ on a torus which has the embedding function,
\begin{equation}
\iota(\theta, \phi)=\left(\begin{array}{c}{(R+r\cos \theta) \cos \phi} \\ {(R+r\cos \theta) \sin \phi} \\ {r\sin \theta}\end{array}\right), \quad \theta \in[0,2 \pi),\quad \phi\in[0,2\pi),
\label{eq_emb}
\end{equation}
where $R$ is the distance from the center of the tube to the center of torus and $r$ is the distance from the center of the tube to the surface of the tube with $r<R$. The induced Riemannian metric is
$$
g_{\textbf{x}^{-1}(\theta, \phi)}(v, w)=v^{\top}\left(\begin{array}{cc}r^2 & 0 \\ 0 & (R+r\cos \theta)^{2}\end{array}\right) w, \quad \forall v, w \in T_{\textbf{x}^{-1}(\theta, \phi)} M.
$$
In this example, we choose $R=2$ and $r=1$. We first demonstrate the performance of DeepONet on different types of $\kappa$ and the sensitivity of the method for several choices of $N_{OBS}$. Subsequently, we employ Diffusion Maps and RBF approaches to approximate surface differential operators in PI-DeepONet, and compare them with DeepONet. In these experiments, we fix $N = 2500$ locations on the torus for the trunk net, and $m=676$ sensors for the branch net. 

\subsubsection{Results from DeepONet}
\label{sec:torus_deeponet}
\comment{We test different types of input function $\kappa$ using vanilla DeepONet, which is pure data-driven. The training data is generated using Diffusion Maps method. }
We test DeepONet with four different types of input function $\kappa$ as followings (Fig.~\ref{fig:torus_deeponet_linear}A):
\begin{enumerate}
    \item Linear function: $\kappa(x, y) = ax + by + 6 +c.$
    \item Exponential function: $\kappa(x, y) = ae^x + be^y + c.$
    \item Piecewise linear function:
    \[ \kappa(x, y) =
      \begin{cases} 
      a_1x + b_1y + 10 & x\leq 0, y\leq 0 \\
      a_1x + b_1y + 10 & x > 0, y\leq 0 \\
      a_1x + b_2y + 10 & x\leq 0, y> 0 \\
      a_2x + b_2y + 10 & x > 0, y> 0 \\
      \end{cases}.\]
    \item Nonlinear function: $\kappa(x, y) = a_1x^2 + b_1y^2 + a_2x + b_2y + c.$
\end{enumerate}

Notice that for points $\textbf{x}=(x,y,z)$ on the torus, we have $-3\leq x,y\leq 3$. And the coefficients $a, b, c, a_1, b_1, a_2, b_2$ are random constants with constraints such that all $\kappa$ are positive on $M$. Since the analytical solution operator for the PDE problem in \eqref{eq1} is not available, the training data is generated by solving a linear problem where the differential operator is approximated by the Diffusion Maps method. 

We train DeepONet on datasets induced by each of these four types of $\kappa$ and a mixed dataset of $\kappa$. Each dataset consists of $N_{OBS}=1000$ randomly chosen $\kappa$, where the randomness is through the coefficients as discussed above. For the mixed dataset, we use 250 for each type, culminating in 1000 randomly chosen $\kappa$. We use a set of $N_t = 1000$ randomly chosen $\kappa$ that is independent to the training data for testing. 
To verify the sensitivity to $N_{OBS}$, we also generate 500 and 100 different linear $\kappa$ for training.

Our experiments show that DeepONet demonstrates robust performance across all scenarios. The model achieves $L^2$ relative errors around 2.5\% for all scenarios (Table~\ref{tab:deeponet_res}). Notably, even in the most complex case involving mixed types of $\kappa$, the model maintains an $L^2$ relative error of 2.67\%. For the simplest case with linear $\kappa$, DeepONet achieves an $L^2$ relative error of 2.07\% when training is performed with $N_{OBS} = 1000$ samples. 

For the results of the linear case (Table~\ref{tab:deeponet_res}), with more training data, the performance of the vanilla DeepONet is better, as expected. Even if we reduce the size of training data $N_{OBS}$ to 100, the error is still small (about 3\%). In Fig.~\ref{fig:torus_deeponet_linear}B, we show an example of one linear $\kappa$ and the solution generated by Diffusion Maps is in Fig.~\ref{fig:torus_deeponet_linear}C. We show the corresponding predictions using $N_{OBS} = 1000, 500,$ and 100 in Fig.~\ref{fig:torus_deeponet_linear}D. It is observed that the absolute error of the prediction is smaller when using a larger size training dataset.

\begin{table}[htbp]
    \footnotesize
    \caption{\textbf{Learning the solution operator on torus using DeepONet.} We report the $L^2$ relative error ($\mathcal{E}$) of our method on four types of $\kappa$. For all cases, we use $N=2500$ point locations.}
    \label{tab:deeponet_res}
    \centering
    \begin{tabular}{ccc}
    \toprule
     Type of $\kappa$ & $N_{OBS}$ & $\mathcal{E}$ \\
    \midrule
    \multirow{3}{*}{Linear} 
      & 1000  & 2.07\%  \\
      &  500  & 2.28\% \\
      &  100  & 3.03\%\\
    Exponential & 1000 & 2.29\%\\
    Piecewise & 1000 & 2.56\%\\
    Higher order & 1000 & 2.58\% \\
    Mixed & 1000 & 2.67\% \\
    \bottomrule
\end{tabular}
\end{table}

\begin{figure}[htbp]
    \centering
    \includegraphics[width=13cm]{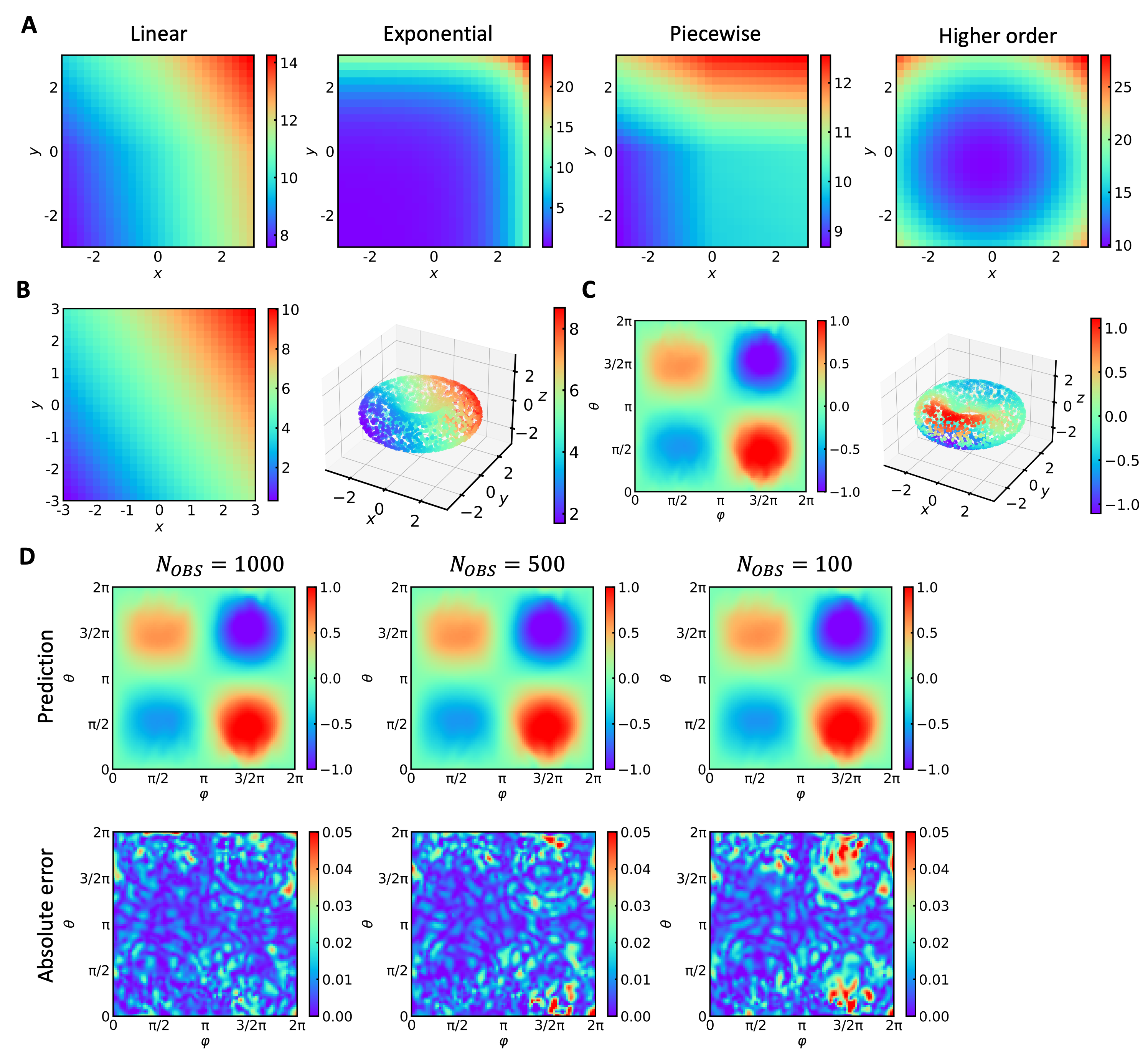}
    \caption{\textbf{DeepONet predictions for the linear $\kappa$ in Section~\ref{sec:torus_deeponet}.}
    (\textbf{A}) Examples of linear $\kappa$, exponential $\kappa$, piecewise $\kappa$, and higher order $\kappa$.(\textbf{B}) 2D and 3D visualizations of one linear $\kappa$. (\textbf{C}) 2D and 3D visualizations of the solution $u$ generated by Diffusion Maps. (\textbf{D}) Predicted solutions and the absolute error with $N_{OBS}=1000, 500$, and $100$. }
    \label{fig:torus_deeponet_linear}
\end{figure}

\subsubsection{Results from PI-DeepONet}
\label{sec:torus_pideeponet}

%We then consider incorporate physics into the DeepONet, and learn the solution operator $\mathcal{G}$ using PI-DeepONet. 
In this subsection, we present results from PI-DeepONet.
%We expect that adding physics constraint would improve the performance of the model, particularly under conditions of limited observations, where $N_{OBS}$ is notably small.
Similar to the previous subsection, we apply Diffusion Maps and RBF approaches to approximate the surface differential operators. The corresponding true solutions for the training and testing datasets are generated using the same approaches employed in computing $\mathcal{L}_{PDE}$.
We set the weights $w_{OBS}$ and $w_{PDE}$ at 1 and 0.0001 for the Diffusion Maps approach, and 1 and 0.001 for the RBF approach, respectively. To demonstrate the efficacy of PI-DeepONet, we train it on several $N_{OBS}$ and compare its performance with DeepONet, and we use $N_t = 1000$ samples of linear $\kappa$ for testing. 

We report the results of DeepONet and PI-DeepONet in Table~\ref{tab:pideeponet_res}. For both Diffusion Maps and RBF estimators, incorporating $N_{PDE}=100$ randomly chosen $\kappa$ into the loss function via PI-DeepONet significantly improves model performance. Specifically, when $N_{OBS}=10$, the $L^2$ relative errors reduce from approximately 25\% to about 4\%, which is a multifold increase in accuracy. When $N_{OBS}=25$, the $L^2$ relative errors decrease from around 5\% and 7\% to below 3\% for Diffusion Maps and RBF approaches. 
Notably, these errors are comparable to those obtained by DeepONet with $N_{OBS}=100$.
Additionally, we observed that an increase in $N_{OBS}$ results in lower $L^2$ relative errors. In Fig.~\ref{fig:torus_linear}B, we show examples of the solutions attained by Diffusion Maps and RBF, with a linear $\kappa$ as depicted in Fig.~\ref{fig:torus_linear}A. 
We show the predictions and absolute errors of DeepONet and PI-DeepONet for $N_{OBS}=10$ and $25$ using Diffusion Maps and RBF estimators in Fig.~\ref{fig:torus_linear}C and D. Notice the dramatic error reduction that can be observed qualitatively across the domain when the physics are incorporated.

\begin{table}[htbp]
    \footnotesize
    \caption{\textbf{Learning the solution operator on torus using DeepONet and PI-DeepONet  with Diffusion Maps or RBF estimator.}}
    \label{tab:pideeponet_res}
    \centering
    \begin{tabular}{ccccc}
    \toprule
     Estimator & Method & $N_{PDE}$ & $N_{OBS}$ & $\mathcal{E}$ \\
    \midrule
    \multirow{4}{*}{Diffusion Maps}
    & \multirow{3}{*}{DeepONet}
      & 0 & 100 & 3.03\% \\
      & & 0 & 25 & 4.72\% \\
      & & 0 & 10 & 26.54\% \\
    & \multirow{2}{*}{PI-DeepONet} 
      & 100 & 25 & \textbf{2.87\% }\\
      & & 100 & 10 & \textbf{4.83\% }\\
     \hdashline
    \multirow{4}{*}{RBF}
    & \multirow{3}{*}{DeepONet} 
      & 0 & 100 & 2.89\% \\
      & & 0 & 25 & 7.04\% \\
      & & 0 & 10 & 25.69\% \\
    & \multirow{2}{*}{PI-DeepONet}  
      & 100 & 25 & \textbf{2.66\%}\\
      & & 100 & 10 & \textbf{3.04\%}\\
    \bottomrule
\end{tabular}
\end{table}

\begin{figure}[htbp]
    \centering
    \includegraphics[width=13cm]{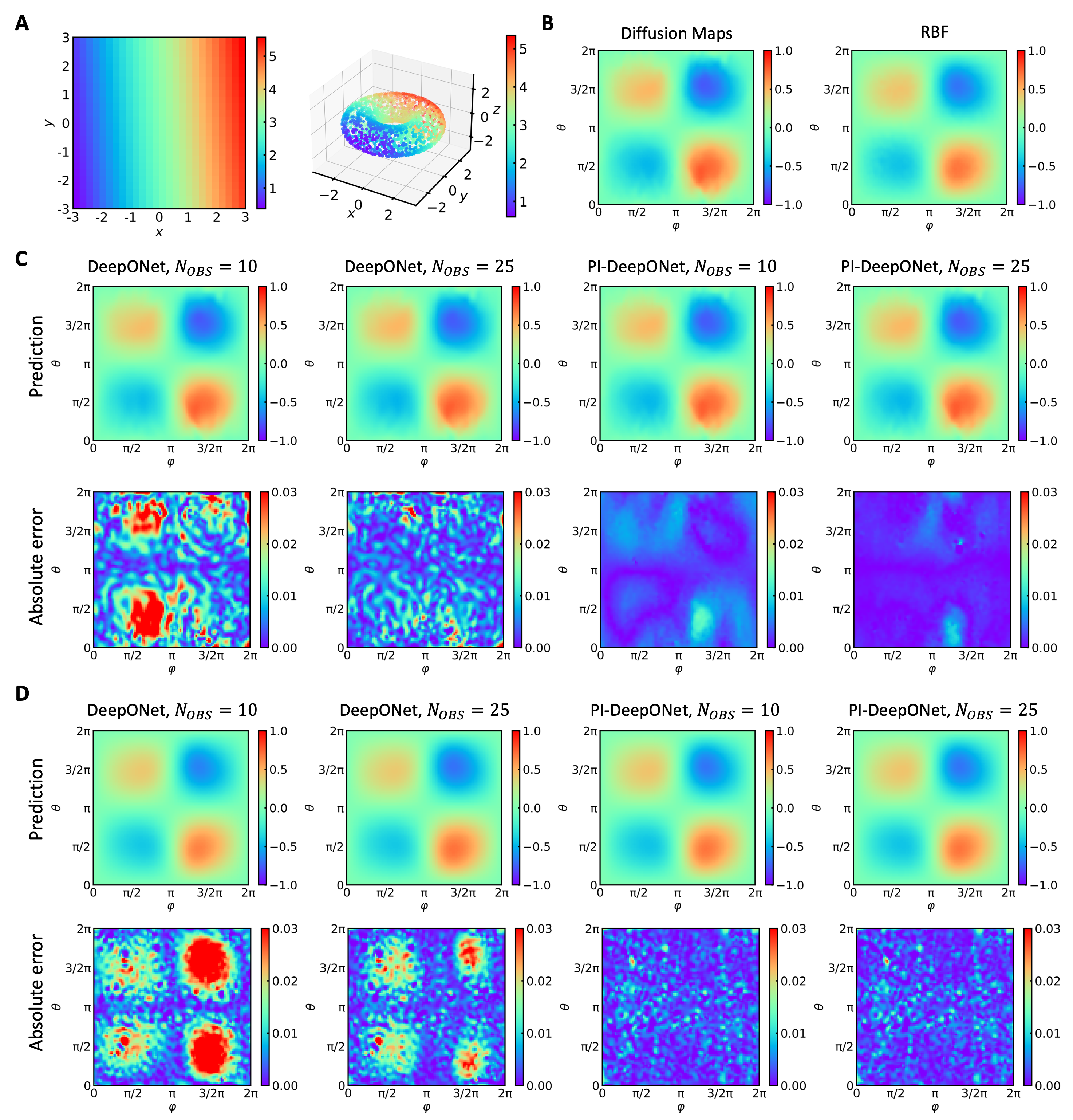}
    \caption{\textbf{An example of prediction of DeepONet and PI-DeepONet in Section~\ref{sec:torus_pideeponet}.} (\textbf{A}) The 2D and 3D visualizations of $\kappa$. (\textbf{B}) The 2D visualization of solutions generated by Diffusion Maps and RBF. (\textbf{C}) The prediction and absolute error of the methods with $N_{OBS}=10$ and 25 for observation loss using Diffusion Maps approach. (\textbf{D}) The prediction and absolute error of the methods with $N_{OBS}=10$ and 25 for observation loss using RBF approach.}
    \label{fig:torus_linear}
\end{figure}

\subsection{Second-order linear elliptic PDE on semi-torus with Dirichlet boundary conditions}
\label{sec:linear_semitorus}

In this section, we consider solving the equation (\ref{eq1}) on a semi-torus with a Dirichlet boundary condition at $\phi = 0, \pi$. The embedding function is the same as (\ref{eq_emb}), except that the range of $\phi$ is [$0, \pi$]. Here, we consider linear $\kappa$ as defined in Section~\ref{sec:torus_deeponet}. The boundary condition is enforced using the points that are sufficiently close to the boundary as described in Section~\ref{sec:PIDeepONet}. We train both DeepONet and PI-DeepONet, where we used GMLS in the latter one to approximate the differential operators since this approximation is simpler and robust for arbitrary boundary conditions. Again, we use $N = 2500$ locations on the torus for the trunk net, and $m=676$ sensors for the branch net as before. The weights $w_{OBS}$ and $w_{PDE}$ are chosen to be 1 and $10^{-7}$, respectively.

When using DeepONet, we observe that the $L^2$ relative errors using $N_{OBS} = 25, 10$ are lower than 0.6\% (Table~\ref{tab:pideeponet_res_semi_torus}), which demonstrates satisfactory accuracy. With only $N_{OBS}=2$, the $L^2$ relative error achieves around 6\%.
The PI-DeepONet improves the accuracy. It's $L^2$ relative errors for $N_{OBS}=25$ and $N_{OBS}=10$ are smaller than those obtained with vanilla DeepONet. Even with only $N_{OBS}=2$, PI-DeepONet outperforms DeepONet and achieves an error less than $3\%$. Detailed qualitative comparisons of the numerical models are shown in Fig.~\ref{fig:semitorus_linear}.

\begin{table}[htbp]
    \caption{\textbf{Learning the solution operator on the semi-torus example using DeepONet and PI-DeepONet with GMLS estimator.}}
    \label{tab:pideeponet_res_semi_torus}
    \centering
    \begin{tabular}{cccc}
    \toprule
     Method & $N_{PDE}$ & $N_{OBS}$ & $\mathcal{E}$ \\
    \midrule
    \multirow{2}{*}{DeepONet} 
      & 0 & 25 & 0.42\% \\
      & 0 & 10 & 0.57\% \\
      & 0 & 2 & 6.17\% \\
    \hdashline
    \multirow{2}{*}{PI-DeepONet} 
      & 100 & 25 & \textbf{0.39\% }\\
      & 100 & 10 & \textbf{0.48\% }\\
      & 100 & 2 & \textbf{2.30\% }\\
    \bottomrule
\end{tabular}
\end{table}

\begin{figure}[htbp]
    \centering    
    \includegraphics[width=13cm]{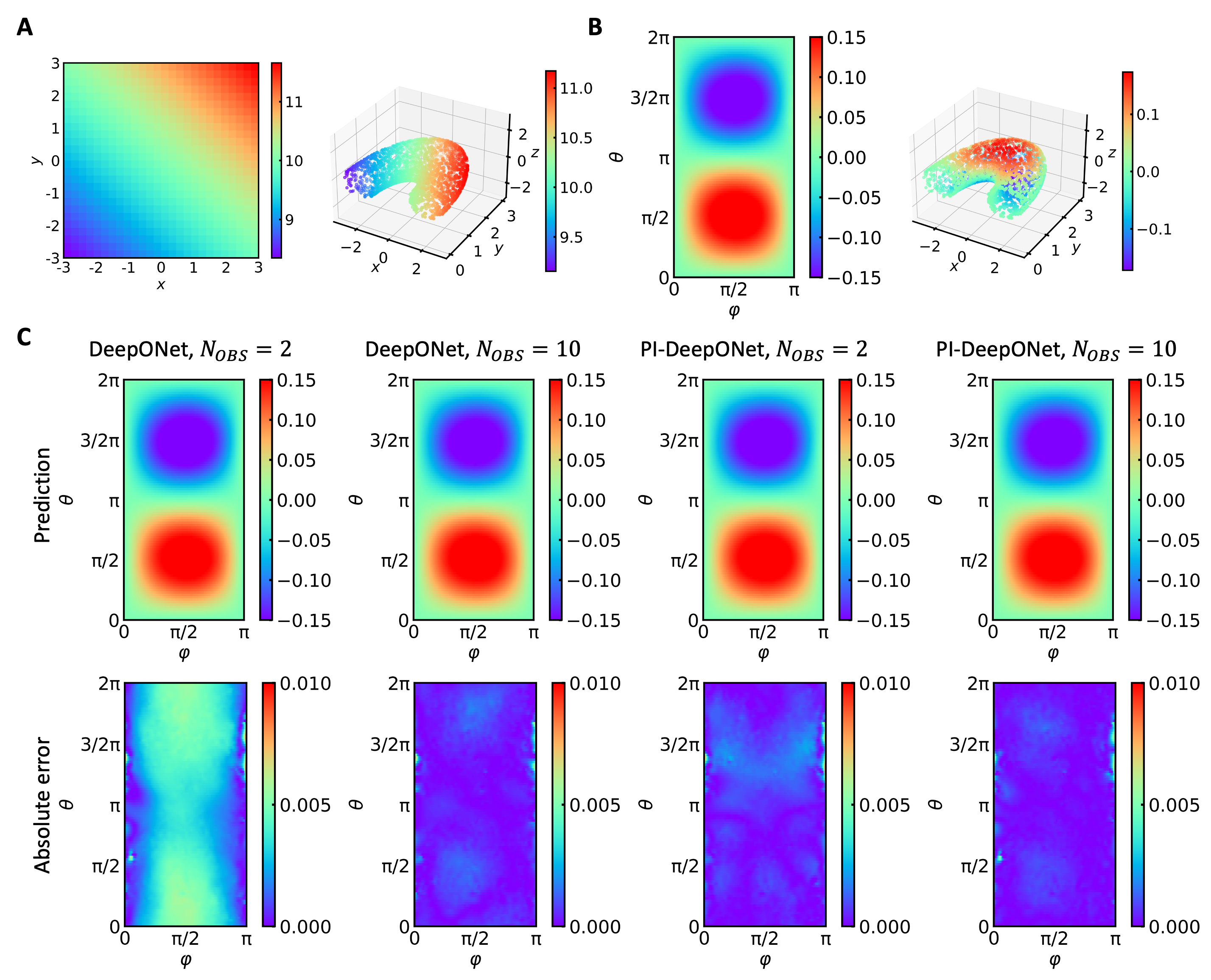}
    \caption{\textbf{Predictions of DeepONet and PI-DeepONet with GMLS approach in Section~\ref{sec:linear_semitorus}.} (\textbf{A}) The 2D and 3D visualizations of $\kappa$. (\textbf{B}) The 2D and 3D visualizations of solution generated by GMLS. (\textbf{C}) The prediction and absolute error of the methods with $N_{OBS}=2$ and $10$ for observation loss.}
    \label{fig:semitorus_linear}
\end{figure}

\subsection{Nonlinear PDE on torus}
\label{sec:nonlinear_torus}

In this section, we consider the following nonlinear (or semilinear) PDE defined on torus:  
$$
-\text{div}_g(\kappa \text{grad}_g u)+u=f(u, \kappa), \quad
f(u, \kappa) = \frac{3}{2}u^2 + u + 
2\kappa u-\frac{1}{2}\kappa^2.
$$
The true solution is set to be $\kappa=a(R+r\cos\theta)$, $u=a\cos\theta $, where $a$ is a constant. We use Diffusion Maps to approximate the surface differential operator. The embedding function and Riemannian metric of the torus is the same as that in Section~\ref{sec:linear_torus}. We use $N = 2500$ locations on the torus for the trunk net, and $m=676$ sensors for the branch net.
We train both DeepONet and PI-DeepONet with $N_{OBS} = 2$ and 10. 

Both DeepONet and PI-DeepONet achieve $L^2$ relative errors around 1\% (Table~\ref{tab:nonlinear_res}). Similar to previous examples, PI-DeepONet performs better on all the cases. A detail qualitative comparison is shown in Fig.~\ref{fig:nonlinear_torus}, where dramatic improvement of the prediction can be observed when the physics is incorporated in training for $N_{OBS}=2$.  

\begin{table}[htbp]
    \footnotesize
    \caption{\textbf{Learning the nonlinear solution operator on torus using DeepONet and PI-DeepONet with diffusion map method.}}
    \label{tab:nonlinear_res}
    \centering
    \begin{tabular}{ccccc}
    \toprule
     Method & $N_{PDE}$ & $N_{OBS}$ & $\mathcal{E}$ \\
    \midrule
    \multirow{2}{*}{DeepONet} 
      & 0 & 10 & 0.89\% \\
      & 0 & 2 & 1.33\% \\
    \hdashline
    \multirow{2}{*}{PI-DeepONet} 
      & 100 & 10 & \textbf{0.69\%}\\
      & 100 & 2 & \textbf{0.71\%}\\
    \bottomrule
\end{tabular}
\end{table}

\begin{figure}[htbp]
    \centering
    \includegraphics[width=13cm]{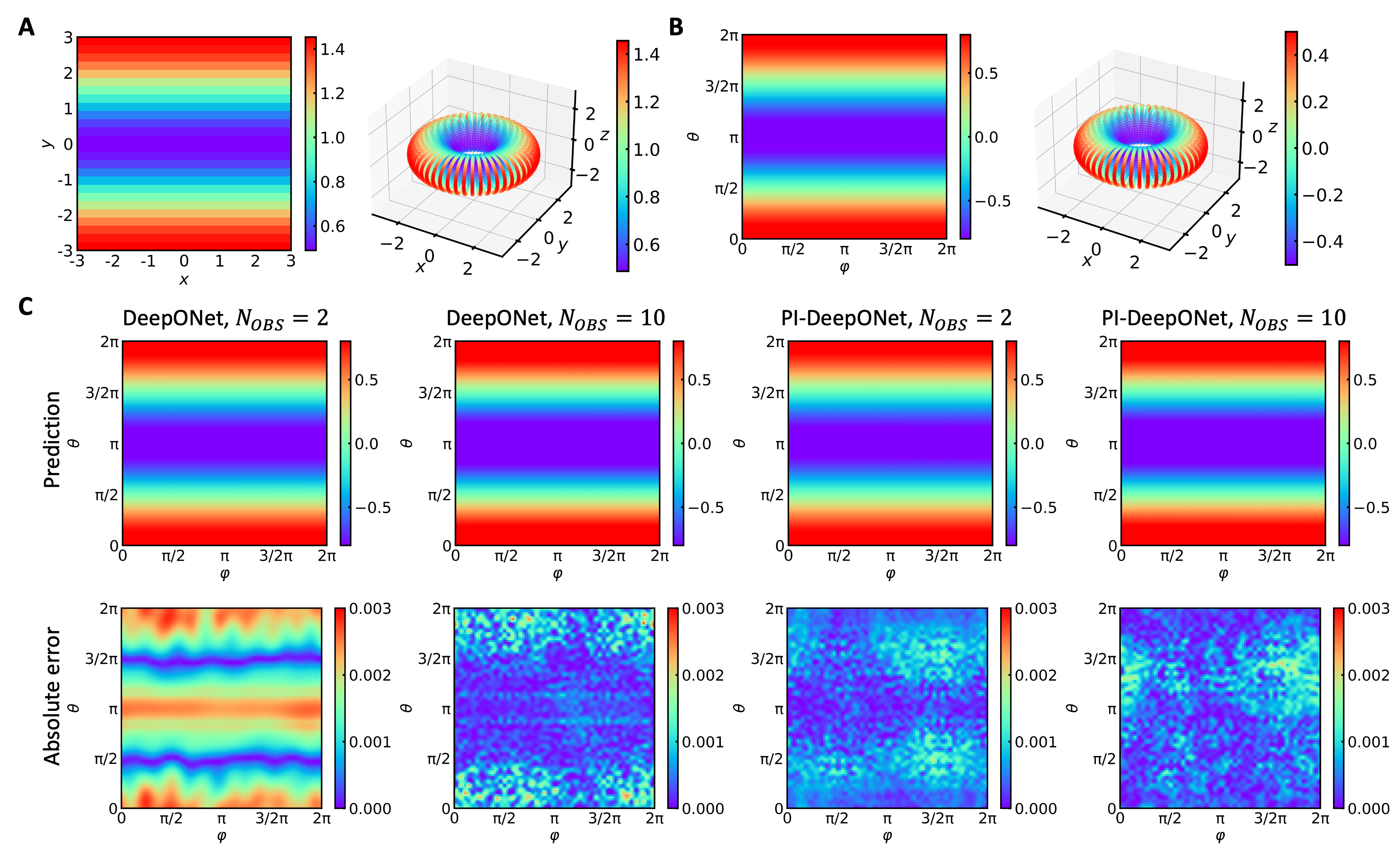}
    \caption{\textbf{Predictions of DeepONet and PI-DeepONet with Diffusion Maps approach in Section~\ref{sec:nonlinear_torus}.} (\textbf{A}) The 2D and the 3D visualizations of $\kappa$.  (\textbf{B}) The 2D and the 3D visualizations of the solution generated by Diffusion Maps. (\textbf{C}) The prediction and absolute error of the methods with $N_{OBS}=2$ and $10$.}
    \label{fig:nonlinear_torus}
\end{figure}

\section{Application to solving Bayesian inverse problems}
\label{sec:inverse_torus}

We consider studying the inverse problem of determining the diffusion coefficient $\kappa$ in elliptic equation (\ref{eq1}) on a torus in Section~\ref{sec:linear_torus} from noisy observations of the solution $u$. We adopt the Bayesian approach to the inverse problem following the method in \cite{harlim2020kernel} and \cite{harlim2022graph}. Instead of using the local kernel method to approximate the forward map as in \cite{harlim2020kernel,harlim2022graph}, we propose to use the PI-DeepONet model. The advantages of this lies in the significant improvement of computational time and flexibility of the neural network model. We will describe the Bayesian formulation of the inverse problem with local kernel method and PI-DeepONet method in Section~\ref{sec:inverse_method}. Then we show the inference results and computational time in Section~\ref{sec:inverse_res}.

\subsection{Bayesian approach to the inverse problem}
\label{sec:inverse_method}

In this section, we consider an application to solve a Bayesian inverse problem involving the elliptic PDE in \eqref{eq1}.
Particularly, our goal is to estimate the diffusion coefficient $\kappa$ in \eqref{eq1} given noisy measurements of $u$ of the form
$$
v(X) = \mathcal{D}u + \eta, \quad \quad \eta \sim \mathcal{N}(0, \Gamma),
$$
at given locations $X=\{\mathbf{x}_i\}_{i=1}^N\subset M$, where the observation map $\mathcal{D}: L^2 \rightarrow \mathbb{R}^J$ is defined as $\mathcal{D} u = u(X) \in \mathbf{R}^N$ and $\Gamma\in \BR^{N\times N}$ is a strictly symmetric positive definite covariance matrix.  

%The distribution for the observation noise is Gaussian, $\eta \sim \mathcal{N}(0, \Gamma)$. 

%The Bayesian solution to this inverse problem is to seek for the posterior distribution of the variable $\kappa$ conditioned on the observed data $y$, where the Bayesian formula is imposed with a choice of prior. 

Here we set $\kappa = e^{\alpha}$, where $\alpha \in (-\infty, \infty)$. We then define a forward map $\mathcal{F}: \alpha = \log(\kappa) \mapsto u$, that takes $\alpha$ to the solution of the PDE in \eqref{eq1}. %For convenience, we also define the map $\mathcal{G}_{\alpha}:= \mathcal{D} \circ \mathcal{F}$ as $\mathcal{G}_{\alpha}: \alpha \mapsto (u(\mathbf{x}_1),\ldots, u(\mathbf{x}_N))$. 
With these notations, the Bayesian inversion problem is to approximate the posterior distribution $\pi^v$ that satisfies the Bayesian formula,
\begin{equation}
\frac{d\pi^v}{d\pi}(\alpha) \propto \exp(-\frac{1}{2} \left| v(X) - \mathcal{D}\circ \mathcal{F}(\alpha)\right|^2_{\Gamma}),\label{likelihood}
\end{equation}
where $\pi$ denotes a prior distribution to be specified, and the right-hand term is the Gaussian likelihood function induced by the distribution of the noise $\eta$. In the formula above, we denote $|y|^2_\Gamma = y^\top \Gamma^{-1}y$ for any $y\in \BR^N$. To realize this estimation problem, we will proceed with a Markov Chain Monte Carlo (MCMC) algorithm to sample from the posterior distribution $\pi^v$. Subsequently, we will use the empirical mean statistics from these posterior samples to estimate $\kappa = e^{\alpha}$ and the empirical variance as a metric that quantifies the uncertainty of the mean estimator.

To realize this goal, we
\begin{enumerate}
    \item    
 Specify a prior distribution $\pi$ for the unknown PDE input $\kappa$ and discretize the prior distribution as $\pi_N$ on the point cloud $X$.
    \item Employ the PI-DeepONet model to approximate the forward map $\mathcal{F}$. 
    \item Employ the graph preconditioned Crank–Nicolson (pCN), an MCMC algorithm to sample from the posterior distribution.
\end{enumerate}

\paragraph{Prior specification and discretization} We consider the Mat\'ern type prior from a two-parameter family of Gaussian measures on $L^2$ (following \cite{harlim2020kernel}). %and consider the point cloud approximation following \cite{harlim2020kernel} Section 2.3.1. 
Particularly, given the point cloud data $X=\{\mathbf{x}_i\}_{i=1}^N\subset M$, we consider the discretized prior distribution
\begin{equation}
\pi(\alpha) \approx \pi_N(\alpha) = \mathcal{N}(0, \mathcal{C}^N_{\tau,s}), \quad \mathcal{C}^N_{\tau,s} = c_N(\tau)(\tau I + \Delta_N)^{-s},
\end{equation}
where $\Delta_N \in \mathbb R^{N\times N}$ is a symmetric graph Laplacian 
matrix 
constructed as described in Eq.~(2.13) of \cite{harlim2020kernel}
on $\{\mathbf{x}_i\}_{i=1}^N$, $c_N(\tau)$ is a normalizing constant, and $\tau>0$, $s>\frac{d}{2}$ are two free parameters. 
%% what is d here
By the karhunen-Lo\'eve expansion, samples from $\pi_N$ can be represented as
\begin{equation}
\alpha_N = c_N(\tau)^{1/2}\sum_{i=1}^N (\tau + \lambda_i^{N})^{-s}\xi_i\varphi_i^{(N)} \in \BR^N,
\label{eq:prior}
\end{equation}
where 
\begingroup
\makeatletter
\def\f@size{9}
\check@mathfonts
$\{(\lambda_i^{(N)}, \varphi_i^{(N)})\}_{i=1}^N$ 
\endgroup
are the eigenvalue-eigenvector pairs for $\Delta_N$ and $\xi_i \overset{\text{i.i.d.}}{\sim} \mathcal{N}(0,1)$. Intuitively, $s$ is the rate of decay of the coefficients and $\tau$ is a length-scale parameter that affects the amplitude of the samples. The normalizing constant $c_N(\tau)$ affects the amplitude of the samples and we choose $c_N{(\tau)}=\frac{N}{\sum_{i=1}^N (\tau + \lambda_i^{N})^{-s}}$ so that $\alpha_N$ has a unit variance. 

\paragraph{Approximation of the forward map} 
The key point of this entire section is to evaluate the effectiveness and accuracy of the PI-DeepOnet in this application. Particularly, in each MCMC step, 
for each proposal (i.e., sample) drawn as prescribed in \eqref{eq:prior}, we will evaluate the likelihood function in \eqref{likelihood}, which required us to compute $\mathcal{F}(\alpha_N)$. Here, we will consider the PI-DeepONet approximation $\mathcal{F}_\theta$ to $\mathcal{F}$ as an alternative to directly solve the PDE in \eqref{eq1} for the input  $\kappa_N := e^{\alpha_N}$.
Computationally, while the PI-DeepONet requires an expensive training procedure that can be done in an offline manner (once and for all), it is computationally attractive for the sequential evaluation of the likelihood function on the proposed samples especially when the number of iterations in the MCMC is large as it is typically encountered in applications. In our context, this computational gain is especially more attractive than directly solving the PDE in \eqref{eq1} with $\kappa_N = e^{\alpha_N}$ on each iteration. If we use the local kernel method (as is done in \cite{harlim2020kernel}), this requires solving an $N\times N$ linear problem in each MCMC iteration. Specifically, the local kernel approximation to $\mathcal{F}$ corresponding to the PDE problem in \eqref{eq1} is given by the discretized forward map $\mathcal{F}_{\varepsilon, N}: \alpha_N \mapsto u_N,$ where $u_N = (L^{\kappa}_{\varepsilon, N} + cI)^{-1}f_N$ and $L^{\kappa}_{\varepsilon, N}$ is defined in Eq.~(2.1) of \cite{harlim2020kernel}. In the numerical results shown below, we will compare the accuracy and efficiency of the PI-DeepONet approximation to those obtained from the local kernel method.

\comment{\[
\frac{d\pi^y_N}{d\pi_N}(\alpha_N) \propto \exp(-\frac{1}{2} \mid y - \mathcal{G}_{\varepsilon, N}(\alpha_N)\mid^2_{\Gamma}),
\]
where $\mathcal{G}_{\varepsilon, N} = D_N \circ \mathcal{F}_{\varepsilon, N}$ is the approximation of $\mathcal{G}_{\alpha}$.} 

%The capability of this local kernel method for solving inverse problem on a torus is demonstrated in \cite{harlim2020kernel} Section 4.3. However, the computation of the operator approximation and the pseudo-inverse is expensive, especially for large $N$ (Remark 4.1 \cite{harlim2020kernel}). Hence, we transform $\alpha_N$ to $\kappa_N = (\kappa(\textbf{x}_1), \ldots, \kappa(\textbf{x}_N)) = e^{\alpha_N}$ and 

To train the PI-DeepONet model, $\mathcal{F}_{\theta}: \alpha_N \mapsto u_N $ to approximate the forward map $\mathcal{F}$, we generate the training dataset with $\kappa_N = e^{\alpha_N}$, where $\alpha_N$ are sampled from the prior distribution (\ref{eq:prior}). We utilize the Diffusion Maps approach to approximate the differential operator and obtain the reference solution. In this numerical experiment, we set $N_{OBS} = N_{PDE}$ for both $\mathcal{L}_{OBS}$ and $\mathcal{L}_{PDE}$, and set $N=m$ locations for both branch net and trunk net, set the weights $w_{OBS}=1$ and $w_{PDE}=0.01$, and use 20000 epochs.

\comment{
We save the model and use it to obtain the prediction $u_N$. Using this approach, we have the discretized posterior distribution $\pi^y_N$:
\[
\frac{d\pi^y_N}{d\pi_N}(\alpha_N) \propto \exp(-\frac{1}{2} \mid y - \mathcal{G}_{\theta}(e^{\alpha_N})\mid^2_{\Gamma}).
\]
}

\paragraph{MCMC with pCN}

The graph pCN algorithm is used to obtain samples from the posterior $\pi^v_N$. Each iterate of the pCN (Algorithm 1 in \cite{harlim2022graph}) proceeds as follows:
\begin{enumerate}
\item Suppose that $\alpha_{N}^{(m-1)}$ is the parameter value at the previous iterate. Then (at the $m$th iterate), draw a proposal as follow,
\[
\tilde{\alpha}_N = (1-\beta^2)^{1/2} {\alpha}_N^{(m-1)} + \beta \gamma_N, 
\]
where $\beta \in (0, 1)$ is a tuning parameter that controls the size of the proposed moves of the chain, and $\gamma_N \sim \pi_N$ is drawn in accordance to \eqref{eq:prior}.
\item Accept the proposal, i.e., set $\alpha_N^{(m)} = \tilde{\alpha}_N$ with probability,
\[ a\left(\alpha_N^{m-1},\tilde{\alpha}_N\right) =
\min\left\{1, \exp\left(\frac{1}{2}\left|v-\mathcal{D}\circ \mathcal{F}_\theta(\alpha_N^{m-1})\right|_{\Gamma}^2-\frac{1}{2}\left|v-\mathcal{D}\circ \mathcal{F}_\theta(\tilde{\alpha}_N)\right|_{\Gamma}^2\right)\right\}. \]
% \begin{displaymath}
% a\left(\alpha_N^{m-1},\tilde{\alpha}_N\right) =
% \min\left\{1, \exp\left(\frac{1}{2}\left|v-\mathcal{D}\circ \mathcal{F}_\theta(\alpha_N^{m-1})\right|_{\Gamma}^2-\frac{1}{2}\left|v-\mathcal{D}\circ \mathcal{F}_\theta(\tilde{\alpha}_N)\right|_{\Gamma}^2\right)\right\}.\end{displaymath}
If the proposal is rejected, then we set $\alpha_N^{(m)} = \alpha_N^{(m-1)}$. We note that if the local kernel is used, then we use $\mathcal{F}_{\epsilon,N}$ in place of the pre-trained DeepONet estimator $\mathcal{F}_\theta$ in the evaluation of the acceptance rate, $a$.
\end{enumerate}

\subsection{Numerical performance and computational time}
\label{sec:inverse_res}
Based on the posterior samples discussed in the previous section, we estimate $\kappa$ with the empirical posterior mean $\Bar{\kappa}$ obtained by averaging over 5000 samples, and subsequently compute the solution $\Bar{u}$ using the local kernel method from $\Bar{\kappa}$ for different noise levels. We use $L^2$ relative error as the metric to quantify the error of the inferred $\Bar{\kappa}$ and the reconstruction error of $\Bar{u}$. The graph Laplacian for the prior is constructed with 16 nearest neighbors and we empirically set the parameters: $\tau = 0.08$ and $s = 6$. The observation noise is assumed to be uncorrelated with a diagonal covariance matrix, $\Gamma = \sigma^2 I$, where $\sigma = 0.01, 0.05,$ and $0.1$. We present the $L^2$ relative error of $\Bar{\kappa}$ and $\Bar{u}$ using both the PI-DeepONet and local kernel methods and report the computational time per iteration of the pCN method in Table~\ref{tab:res_inverse_torus}. 

\begin{table}[htbp]
    \footnotesize
    \caption{\textbf{The inverse problem on a torus.} We compare the results using local kernel method and PI-DeepONet method.}
    %{\color{red}Would it be possible to make the last two rows to be more comparable? Especially for the case of $50\times 50$, the errors are so far away from 8.08\% to 11.76 \%. I think there must be other parameters to tune in DeepONet that can improve the result. Anran: Got a much better model that is even better than the baseline}
    \label{tab:res_inverse_torus}
    \footnotesize
    \centering
    \begin{tabular}{cccccccc}
    \toprule
     Method & $N_{PDE}$, $N_{OBS}$ & $N$ & $\sigma$ & $\beta$ & $\mathcal{E}$ of $\kappa$ & $\mathcal{E}$ of $u$ & Time (s)\\
    \midrule
    \multirow{2}{*}{Local kernel} 
      & - & $20\times20$ & 0.01 & 0.02 & 7.10\% & 1.87\% & 0.0066\\
      & - & $20\times20$ & 0.05 & 0.01 & 6.01\% & 3.57\% & 0.0066\\
      & - & $20\times20$ & 0.1  & 0.02 & 11.18\% & 6.11\% & 0.0066\\
      & - & $50\times50$ & 0.01 & 0.02 & 5.56\% & 1.49\% & 0.3008\\
    \hdashline
    \multirow{2}{*}{PI-DeepONet} 
       & 1000 & $20\times20$ & 0.01 & 0.02 & 8.67\% & 3.88\% & 0.0008\\
       & 1000 & $20\times20$ & 0.05 & 0.01 & 7.18\% & 4.75\%  & 0.0008\\
       & 1000 & $20\times20$ & 0.1 & 0.02 & 12.04\% & 7.46\%  & 0.0008\\
       % & 300 & $50\times50$, $26\times26$ & 0.01 & 10.03\% & 13.93\% & 0.00166\\
       & 300 & $50\times50$ & 0.01 & 0.02 & 7.24\% & 2.93\% & 0.0160\\
    \bottomrule
\end{tabular}
\end{table}

PI-DeepONet demonstrates comparable results to the local kernel method with significantly reduced computational time. Particularly when $N=60\times60$, the local kernel method's iteration time is approximately 1.22 seconds, which is impractical for large-scale applications. In contrast, the iteration time for PI-DeepONet is only about 0.03 seconds, which is 40 times faster. In terms of accuracy, we suspect that the lack of accuracy of the PI-DeepONet in some cases is due to the interpolation error of the estimator $\mathcal{G}_\theta$, which is possibly due to lack of tuning of the DeepONet hyperparameters such as $N_{OBS}$, $N_{PDE}$, $\lambda$, $w_{OBS}, W_{PDE}$.  

In Fig.~\ref{fig:pideeponet_inverse}, we display the ground truth, prediction, and absolute error of $\kappa$ and standard deviations of posteriors for different noise levels $\sigma$. It shows the effectiveness of PI-DeepONet method in achieving reasonable accuracy with fast computational speeds. 
We find that the wall clock time for each pCN iteration grows approximately quadratically as a function of the number of point cloud data, $N$ (Fig.~\ref{fig:time}). 
The computational cost with PI-DeepONet is more than 8 times faster than that of the local kernel for $N= 20^2$, and about 40 times faster for $N= 60^2$. {The red and blue dashed lines in the figure represent the fitting of the wall-clock time versus $N$ for PI-DeepONet and the local kernel method, respectively. Our empirical results show that the proposed PI-DeepONet method has an increase in wall-clock time of $\mathcal{O}(N^{1.77})$ as the data size grows. It represents a substantial improvement over the local kernel method, which exhibits a time complexity of $\mathcal{O}(N^{2.25})$.}

\begin{figure}[htbp]
    \centering
    \includegraphics[width=12cm]{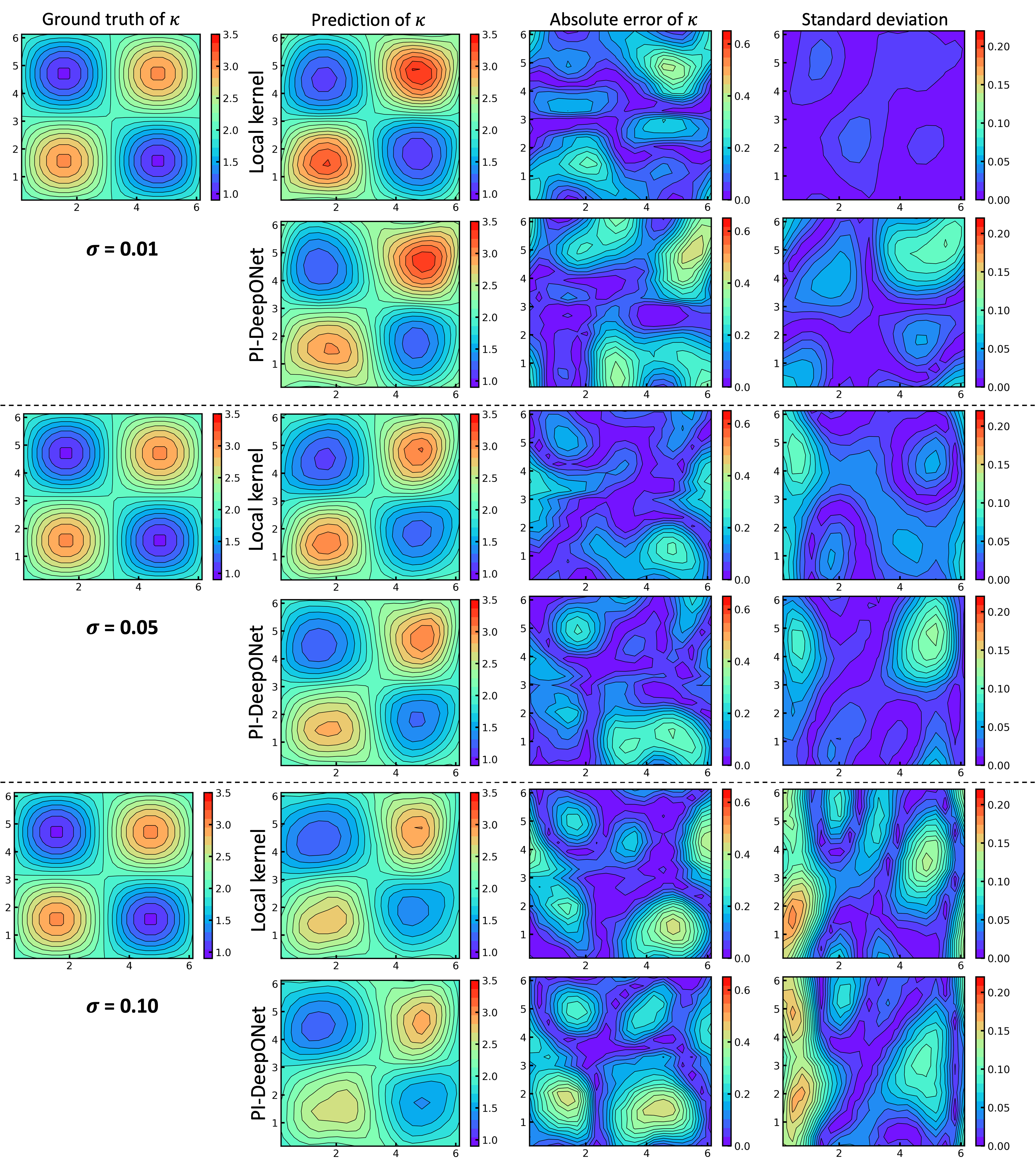}
    \caption{\textbf{The ground truth of $\kappa$, posterior mean of $\kappa$, and its standard deviation of the local kernel and PI-DeepONet method for different noise level $\sigma$.}}
    \label{fig:pideeponet_inverse}
\end{figure}

\begin{figure}[htbp]
    \centering
    \includegraphics[width=7cm]{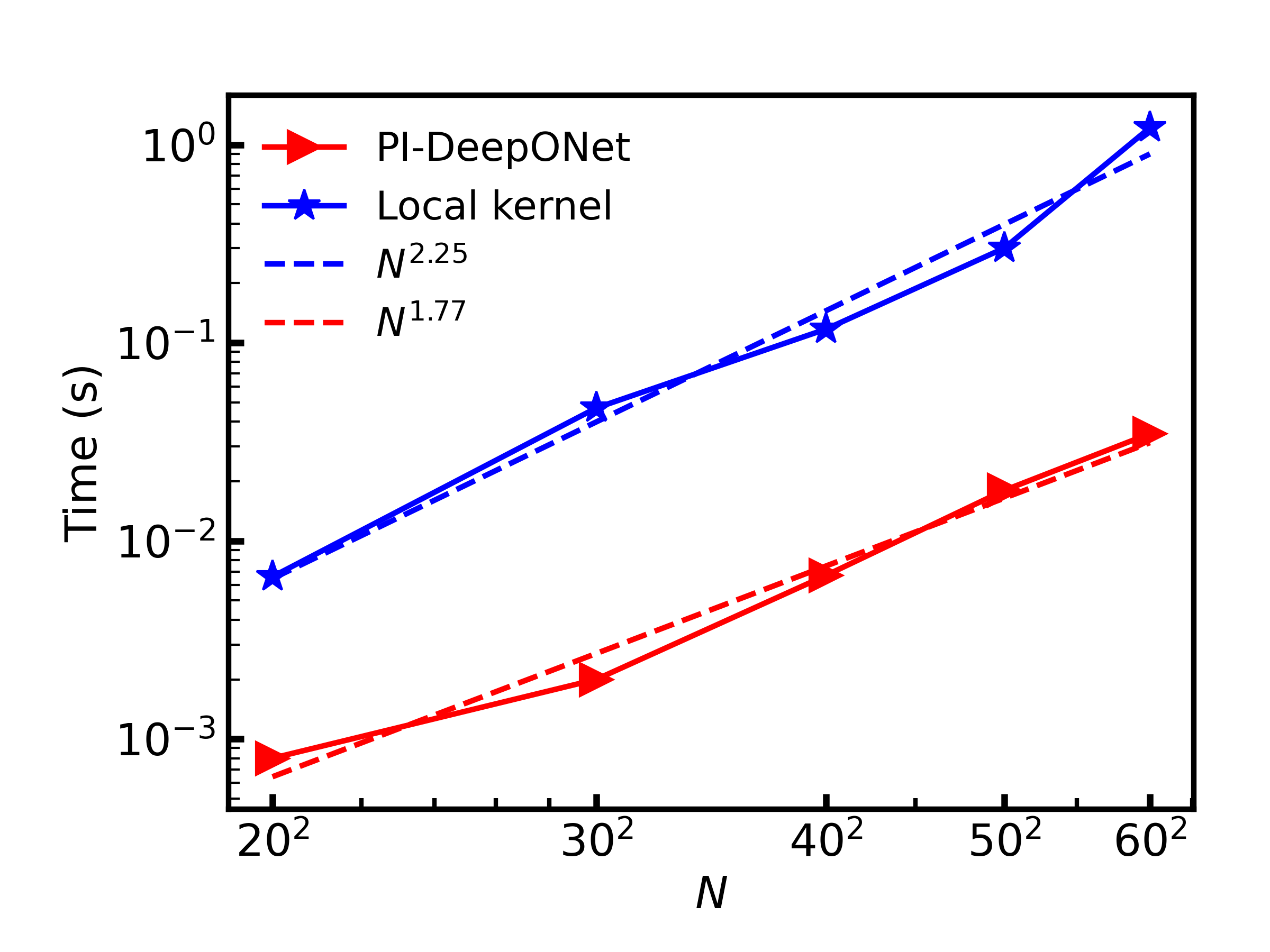}
    \caption{\textbf{Wall-clock time of the local kernel and PI-DeepONet method per pCN iteration as functions of the number of point cloud data, $N$.}}
    \label{fig:time}
\end{figure}

% Here, we note that theoretically the computational cost of local kernel method is dominated by the inversion of the sparse $N\times N$ discrete operator which requires $\mathcal{O}(NK)=\mathcal{O}(N^{1.5}$ operations. This demonstrate the potential of our PI-DeepONet method for large-scale applications.

Here, we analyze the time complexity of the local kernel method and PI-DeepONet that differs in the pCN iteration. The computational cost of the local kernel method is dominated by the inversion of the $N\times N$ discrete operator $L^{\kappa}_{\varepsilon, N}$, which is a sparse matrix with $NK$ nonzeros. In the experiments, we use $K = 1.5\sqrt{N}$ nearest neighbors in the diffusion maps approach. The inversion requires $\mathcal{O}(NK)= \mathcal{O}(N^{3/2})$ operations. In our implementation, we use \texttt{scipy.sparse.linalg.spsolve} function instead of pseudo-inverse or direct inversion in Python, which is more efficient to ensure a fair comparison. For the one-time prediction of PI-DeepONet, the time complexity is dominated by the time it takes to complete a forward pass through the neural network. The input of the branch net is an $N$ dimensional vector, and the input of the trunk net is a $N \times 3$ dimensional vector for the collocation points. For branch net, we use a CNN with two convolutional layers with the filter size $F \times F$ and stride $S$ for both layers. The numbers of filters for the convolutional layers are $K_1$ and $K_2$, and we use two dense layers with $K_3$ output features for both. The input size of the first convolutional layer is $\sqrt{N} \times \sqrt{N} \times 1$, and the output shapes of two layers are $W \times W \times K_1$ and $H \times H \times K_2$, where $W = \lfloor(\sqrt{N} - F)/S \rfloor + 1$ and $H = \lfloor(W - F)/S\rfloor + 1$. Hence, the time complexity of this CNN is $\mathcal{O}(W^2K_1F^2 + H^2K_2F^2 + H^2K_2K_3 + K_3^2)$. In our experiment, $F = 3, S = 2, K_1 = 16, K_2 = K_3 = 32$ are fixed for different $N$. Hence, the complexity of CNN is $\mathcal{O}(N)$. For trunk net, we use an FNN with $L=3$ layers and depth $D = 32$ for each layer, which requires $\mathcal{O}((2 \times 3 \times D + 2 \times L \times D \times D) \times N) = \mathcal{O}(N)$ operations. The combination of the branch net and trunk net is the dot product of two vectors, and the complexity is $\mathcal{O}(N \times (D+D-1)) = \mathcal{O}(N)$. Hence the total time complexity of DeepONet is approximately $\mathcal{O}(N) + \mathcal{O}(N) + \mathcal{O}(N) = \mathcal{O}(N)$.

\comment{
\begin{table}[htbp]
    \caption{\textbf{The time scaling of the local kernel method and PI-DeepONet method with the number of locations $N$ used.}}
    \label{tab:time_res}
    \centering
    \begin{tabular}{cccccc}
    \toprule
     $N$ & $20 \times 20$ & $30 \times 30$ & $40 \times 40$ & $50 \times 50$ & $60 \times 60$ \\
    \midrule
    Local kernel &
    0.0066& 0.0469& 0.1173& 0.3008& 1.2196\\
    PI-DeepONet &
    0.0008& 0.0020& 0.0067& 0.0180& 0.0348\\
    \bottomrule
\end{tabular}
\end{table}
}

\section{Conclusions}
\label{sec: conclusions}

In this paper, we developed DeepONet and PI-DeepONet to solve elliptic PDEs on manifolds and learn the corresponding solution operators. We also integrated a PI-DeepONet model into a Bayesian framework for tackling inverse problems.

We demonstrated the effectiveness of our method on numerical experiments across different scenarios including linear and nonlinear problems, torus and semi-torus, and inverse problems. We have shown that DeepONet and PI-DeepONet perform well on random point cloud with flexibility, and PI-DeepONet with physics incorporated achieve better results than DeepONet. In the inverse problem,  PI-DeepONet is able to reduce computational time significantly while maintaining comparable accuracy. For the future work, we aim to utilize our method on more complex PDEs and exploring their effectiveness in more irregular geometries.

\section*{Acknowledgments}

This work was supported by the U.S. Department of Energy Office of Advanced Scientific Computing Research under Grant No.~DE-SC0022953 and the U.S. National Science Foundation under Grant No.~DMS-2347833 as part of the Joint DMS/NIGMS Initiative to Support Research at the Interface of the Biological and Mathematical Sciences. The research of J.H. was partially supported by the NSF grants DMS-2207328, DMS-2229535, and the ONR grant N00014-22-1-2193.  

\appendix
\section{Approximation of differential operators on manifolds}\label{AppA}
In this section, we provide a short overview of three differential operator estimators on manifolds: Diffusion Maps methods, Radial Basis Function (RBF) method, and Generalized Moving Least-Squares (GMLS) method.

\subsection{Review of Diffusion Map method}\label{sec:DM}
Here, we provide a short review of a kernel approximation to the differential operator $\mathcal{L}^\kappa :=-\text{div}_g(\kappa\textup{grad}_g \,)$ with a fixed-bandwidth Gaussian kernel \cite{gh2019,harlim2020kernel,jiang2020ghost} on a $d$-dimensional manifold $M$. For more accurate estimation when the data is non-uniformly distributed, one can extend this graph-Laplacian approximation using the variable bandwidth kernel as in \cite{bh:16vb}.

Assume that we are given a set of point cloud data $X=\{\textbf{x}_i\}_{i=1}^N\subset M$ independent and identically distributed according to $\pi$. Let $h:[0,\infty)\rightarrow[0,\infty)$ be defined as $h(s)=\frac{e^{-s / 4}}{(4 \pi)^{d / 2}}$ and  $\epsilon>0$ a fixed bandwidth parameter. Following the pointwise estimation method in \cite{liang2024solving}, one can first approximate the sampling density $q=d\pi/dV$ evaluated at $\mathbf{x}_i$, with $Q_i:=\epsilon^{-d / 2} N^{-1} \sum_{j=1}^{N} h\left(\frac{\left\|\mathbf{x}_{i}-\mathbf{x}_{j}\right\|^{2}}{\epsilon}\right)$ and then construct $\mathbf{W}\in\mathbb{R}^{N\times N}$ as
$$
\mathbf{W}_{i j}:=\epsilon^{-d / 2-1} N^{-1} h\left(\frac{\left\|\mathbf{x}_{i}-\mathbf{x}_{j}\right\|^{2}}{\epsilon}\right) \sqrt{\kappa\left(\mathbf{x}_{i}\right) \kappa\left(\mathbf{x}_{j}\right)} Q_{j}^{-1}.
$$

Next, one can obtain the diffusion Maps (DM)
estimator, $\mathbf{L}^{DM}:=\mathbf{D}-\mathbf{W}$, where $\mathbf{D}\in\mathbb{R}^{N\times N}$ is a diagonal matrix with diagonal entries $\mathbf{D}_{i i}=\sum_{j=1}^{N} \mathbf{W}_{i j}$. Then $\mathbf{L}^{DM}$ is a discrete estimator to the operator $\mathcal{L}^\kappa=-\text{div}_g(\kappa\textup{grad}_g \,)$ in high probability. 

For accurate estimation, one has
to specify the appropriate bandwidth parameter, $\epsilon$. In our
implementation, we use the $k-$nearest neighbor (kNN) algorithm to avoid computing
the distances between pairs of points that are sufficiently far away.
That is, for each $x_i$, we only use its $k-$nearest neighbors, denoted by $ x_{i_r}$ for $r=1,\ldots,k$, to construct the kernel in $\mathbf{W}_{ij}$, which reduces the computational cost from $O(N^2)$ to $O(kN)$ in the construction of $\mathbf{L}_\epsilon$, in addition to a one time cost of employing kNN algorithm.

Our choice of $\epsilon$ follows the method originally proposed in
\cite{coifman2008TuningEpsilon}. The idea relies on the following
observation,
\begin{eqnarray}
S(\epsilon):=\frac{1}{Vol(M)^2}\int_M \int_{T_xM} h\left(\frac{\left\|x-y\right\|^{2}}{\epsilon}\right)  dy\,dV(x)
= \frac{1}{Vol(M)^2}\int_M (4\pi\epsilon)^{d/2} dV(x) = \frac{%
(4\pi\epsilon)^{d/2}}{Vol(M)}.  \label{scalingS}
\end{eqnarray}
Since $S$ can be approximated by a Monte-Carlo integral, for a fixed $k$, we
approximate,
\begin{eqnarray}
S(\epsilon) \approx \frac{1}{Nk}\sum_{i=1}^{N}\sum_{r=1}^{k} \exp\Big(-\frac{%
\|x_i-x_{i_r}\|^2}{4\epsilon}\Big),  \notag
\end{eqnarray}
where $\{x_{i_r}\}_{r=1}^k$ is the kNN of each $x_i$. We
choose $\epsilon$ from a domain (e.g., $[2^{-14},10]$ in our numerical
implementation) such that $\frac{d\log(S)}{d\log\epsilon} \approx \frac{d}{2}
$. Numerically, we found that the maximum slope of $\log(S)$ often coincides
with $d/2$, which allows one to use the maximum value as an estimate for the
intrinsic dimension $d$ when it is not available, and then choose the
corresponding $\epsilon$. We should also point out that this bandwidth
tuning may not necessarily give the most accurate result (as noted in \cite%
{bh:16vb}); however, it gives a useful reference value for further tuning.

\subsection{Review of RBF method}\label{AppA:A2}
\par In this section, we review the global Radial Basis Function methods  \cite{piret2012orthogonal,fuselier2013high} to approximate surface operators. We first review the radial basis function interpolation over a set of point cloud data. Let $M$ be a $d$-dimensional smooth manifold of $\mathbb{R}^n$. Given a set of (distinct) nodes $X=\{\textbf{x}_i\}_{i=1}^N\subset M$ and function values $\mathbf{f}:=(f(\textbf{x}_1),\ldots, f(\textbf{x}_N))^\top$\ at $X = \{\textbf{x}_{j}\}_{j=1}^{N}$, where $f:M\rightarrow\mathbb{R}$ is an arbitrary smooth function, a radial basis function (RBF) interpolant takes the form
\BEA
I_{\phi_s}\mathbf{f}(x)=\sum_{j=1}^Nc_j\phi_s\big(\Vert \textbf{x}-\textbf{x}_j\Vert\big), \ \  \textbf{x}\in M,
\label{RBF-form}
\EEA
where $c_j$ are determined by requiring $I_{\phi_s} \mathbf{f}|_X=\mathbf{f}$. Here, we have defined the interpolating operator $I_{\phi_s}:\mathbb{R}^{N} \to C^\alpha(\mathbb{R}^n)$, where $\alpha$ denotes the smoothness of the radial kernel $\phi_s$. In \eqref{RBF-form}, the notation $\Vert\cdot\Vert$ corresponds to the standard Euclidean norm in the ambient space $\mathbb{R}^n$. The interpolation constraints can be expressed as the following linear system
\BEA
\underbrace{\left[
\begin{array}{ccc}
\phi _{s}(\Vert \mathbf{x}_{1}-\mathbf{x}_{1}\Vert ) & \cdots  &
\phi _{s}(\Vert \mathbf{x}_{1}-\mathbf{x}_{N}\Vert ) \\
\vdots  & \ddots  & \vdots  \\
\phi _{s}(\Vert \mathbf{x}_{N}-\mathbf{x}_{1}\Vert ) & \cdots  &
\phi _{s}(\Vert \mathbf{x}_{N}-\mathbf{x}_{N}\Vert )%
\end{array}%
\right]}_{\boldsymbol{\Phi}}\underbrace{\left[\begin{array}{c}c_{1} \\ \vdots \\ c_{N}\end{array}\right]}_{\mathbf{c}}=\underbrace{\left[\begin{array}{c}f\left(\mathbf{x}_{1}\right) \\ \vdots \\ f\left(\mathbf{x}_{N}\right)\end{array}\right]}_{\mathbf{f}},
\label{RBF-inter}
\EEA
where $[\boldsymbol{\Phi}]_{i,j}=\phi_s(\Vert\textbf{ x}_i-\textbf{x}_j\Vert)$.

In literature, many types of radial functions have been proposed for the RBF interpolation. For example, the Gaussian function $\phi_s(r)=e^{-(sr)^2}$, Multiquadric function $\phi_s(r)=\sqrt{1+(sr)^2}$, Wendland function $\phi_s(r)=(1-sr)^m_+p(sr)$ for some polynomials $p$, and the Inverse quadratic function $\phi_s(r)=1/(1+(sr)^2)$. In our numerical examples, we will implement the Inverse quadratic function. While this kernel yields a positive definite matrix $\boldsymbol{\Phi}$, the matrix tends to have a high condition number, especially when the point cloud data is randomly sampled. To overcome this issue, we solve the linear problem in \eqref{RBF-inter} using the standard pseudo-inverse method with an appropriately specified tolerance. Numerically, we use $\text{pinv}(\Phi,\text{1e-6})$ in the MATLAB code. 

Now, we review the RBF projection method proposed in \cite{fuselier2013high} for a discrete approximation of surface differential operators on manifolds. The projection method represents the surface differential operators as tangential gradients, which are formulated as the projection of the appropriate derivatives in the ambient space. For any point $\textbf{x}=(x^1,...,x^n)\in M$, we denote the tangent space of $M$ at $\textbf{x}$ as $T_\textbf{x}M$ and a set of orthonormal vectors that span this tangent space as $\{\textbf{t}_i\}_{i=1}^d$. Then the projection matrix $\textbf{P}$ which projects vectors in $\mathbb{R}^n$ to $T_\textbf{x}M$ could be written as $\textbf{P}=\sum_{i=1}^d\textbf{t}_i\textbf{t}_i^\top$. Subsequently, the surface gradient on a smooth function $f:M\to \mathbb{R}$ evaluated at $\textbf{x}\in M$ in the Cartesian coordinates is given as,
$$
\textup{grad}_g f(\mathbf{x}):=\textbf{P}\overline{\textup{grad}}_{\mathbb{R}^n}f(\mathbf{x})=\left(\sum_{i=1}^d\textbf{t}_i\textbf{t}_i^\top\right)\overline{\textup{grad}}_{\mathbb{R}^n}f(\mathbf{x}),
$$
where $\overline{\textup{grad}}_{\mathbb{R}^n}=[\partial_{x^1}, \cdots, \partial_{x^n}]^\top$ is the usual gradient operator and the subscript $g$ is to associate the differential operator to the Riemannian metric $g$ induced by $M$ from $\mathbb{R}^n$. Let $\mathbf{e}^\ell,\ell=1,...,n$ be the standard orthonormal vectors  in $x^\ell$ direction in $\mathbb{R}^n$, we can rewrite above expression in component form as
$$
\textup{grad}_g f (\mathbf{x}):=\left[\begin{array}{c}\left(\mathbf{e}^{1} \cdot \mathbf{P}\right) \overline{\textup{grad}}_{\mathbb{R}^n} f(\mathbf{x}) \\ \vdots\\ \left(\mathbf{e}^{n} \cdot \mathbf{P}\right) \overline{\textup{grad}}_{\mathbb{R}^n}f(\mathbf{x}) \end{array}\right]
=\left[\begin{array}{c}\mathbf{P}^{1} \cdot \overline{\textup{grad}}_{\mathbb{R}^n}f(\mathbf{x}) \\ \vdots \\ \mathbf{P}^{n} \cdot \overline{\textup{grad}}_{\mathbb{R}^n} f(\mathbf{x})\end{array}\right]:=\left[\begin{array}{c}\mathcal{G}_{1}f(\mathbf{x}) \\ \vdots\\ \mathcal{G}_{n} f(\mathbf{x})\end{array}\right],
$$
where $\textbf{P}^\ell$ is the $\ell-$th row of the projection matrix $\textbf{P}$. 

Then we can approximate any differential operators (related to gradient and divergence) of functions $f$ by differentiating the RBF interpolant. That is, for $\ell=1,...,n$, and $\mathbf{x}_i \in X$,
$$
\mathcal{G}_\ell f(\textbf{x}_i)\approx (\mathcal{G}_\ell I_{\phi_s}\mathbf{f})(\textbf{x}_i)=\sum_{j=1}^Nc_j\mathcal{G}_\ell\phi_s(\Vert \textbf{x}_i-\textbf{x}_j\Vert),
$$
which can be written in matrix form as,
\BEA
\underbrace{\left[\begin{array}{c}\mathcal{G}_\ell f\left(\mathbf{x}_{1}\right) \\ \vdots \\ \mathcal{G}_\ell f\left(\mathbf{x}_{N}\right)\end{array}\right]}_{\mathcal{G}^\ell \mathbf{f}}\approx\mathbf{B}_\ell\underbrace{\left[\begin{array}{c}c_{1} \\ \vdots \\ c_{N}\end{array}\right]}_{\mathbf{c}}=\mathbf{B}_\ell \boldsymbol{\Phi}^{-1}\mathbf{f},\label{G_ell}
\EEA
where $[B_\ell]_{i,j}=\mathcal{G}_\ell\phi_s(\Vert \textbf{x}_i-\textbf{x}_j\Vert)$ and we have used $\boldsymbol{\Phi}$ as defined in (\ref{RBF-inter}) in the last equality. Hence, the differential matrix for the operator $\mathcal{G}_\ell$ is given by
\BEA
\mathbf{G}_\ell:=\mathbf{B}_\ell \boldsymbol{\Phi}^{-1}.
\label{diff-op}
\EEA
Then the operator $\mathcal{L}^\kappa =-\text{div}_g(\kappa\textup{grad}_g \,)$ could be approximated by 
\BEA
\mathbf{L}^{RBF}=-\sum_{\ell=1}^nG_\ell(\kappa G_\ell).
\EEA

\subsection{Review of GMLS approach}\label{App:A3}
In this section, we review the Generalized Moving Least-Squares (GMLS) method for approximating the surface operators. Indeed, there are two different ways to employ the GMLS approach to estimate operators on manifolds: using intrinsic differential geometry as in  \cite{liang2013solving,gross2020meshfree} and using extrinsic differential geometry as in \cite{suchde2019meshfree,jiang2024generalized}. For convenience of discussion, we briefly introduce the second one which will use the extrinsic formulation of differential operators as in Section~\ref{AppA:A2}.

For an arbitrary point ${\textbf{x}_0}\in X\subset M$, we denote its $K-$nearest neighbors in $X$ by $S_{{\textbf{x}_0}}=\{{\textbf{x}_{0,k}}\}_{k=1}^K\subset X$. Denote $\textbf{f}_{{\textbf{x}_0}}=(f({\textbf{x}_{0,1}}),...,f(\textbf{x}_{0,K}))^\top$. First, let $\mathbb{P}_{\textbf{x}_0}^{l,d}$ be the space of polynomials with degree up to $l$ in $d$ variables at the point ${\textbf{x}_0}$, i.e., $\mathbb{P}_{\textbf{x}_0}^{l,d}=\text{span}(\{p_{{\textbf{x}_0},\alpha}\}_{|\alpha|\leq l})$, where  $\alpha=(\alpha_1,..,\alpha_d)$ is the multi-index notation and $p_{{\textbf{x}_0},\alpha}$ is the basis polynomial functions defined as
$$
 p_{{\textbf{x}_0},\alpha}(\textbf{x})= \textbf{z}^{\alpha} =  \prod_{i=1}^d (z^i)^{\alpha_i}= \prod_{i=1}^d \left[ \textbf{t}_{{\textbf{x}_0},i}\cdot(\textbf{x}-{\textbf{x}_0})\right]^{\alpha_i},\quad |\alpha|\leq l.
$$
Here, $ \textbf{t}_{{\textbf{x}_0},i}$ is the $i$th tangent vector at ${\textbf{x}_0}$ for $i=1,...,d$. By definition, the dimension of the space  $\mathbb{P}_{\textbf{x}_0}^{l,d}$ is $m=\left(\begin{matrix}l+d\\ d\end{matrix}\right)$. For $K>m$, we can define an operator $\mathcal{I}_{\mathbb{P}}:\textbf{f}_{{\textbf{x}_0}}\in\mathbb{R}^K\rightarrow \mathcal{I}_{\mathbb{P}}\textbf{f}_{{\textbf{x}_0}}\in \mathbb{P}_{\textbf{x}_0}^{l,d}$ such that $\mathcal{I}_{\mathbb{P}}\textbf{f}_{{\textbf{x}_0}}$ is the optimal solution of the following least-squares problem:%unique solution of the least squares problem:
 \BEA
\underset{q\in \mathbb{P}_{\textbf{x}_0}^{l,d}}{\operatorname{min}}\sum_{k=1}^K\left(f({\textbf{x}_{0,k}})-q ({\textbf{x}_{0,k}})\right)^2.
 \label{eqn:int_LS}
 \EEA
The solution to the least-squares problem (\ref{eqn:int_LS}) can be represented as $\mathcal{I}_{\mathbb{P}}\textbf{f}_{{\textbf{x}_0}}=\sum_{|\alpha|\leq l}b_\alpha p_{{\textbf{x}_0},\alpha}$, where the concatenated coefficients $\textbf{b}=(b_{\alpha(1)},...,b_{\alpha(m)})^\top$ satisfy the normal equation,
 \BEA
(\boldsymbol{\Phi}^\top\boldsymbol{\Phi})\textbf{b}=\boldsymbol{\Phi}^\top\textbf{f}_{{\textbf{x}_0}},
 \label{eqn:inter}
 \EEA
where  \BEA
 \boldsymbol{\Phi}_{kj}=p_{{\textbf{x}_0},\alpha(j)}({\textbf{x}_{0,k}}),\quad 1\leq k\leq K,\ 1\leq j\leq m.
 \label{eqn:matrix_phi}
 \EEA
 Using the notations defined in Section \ref{AppA:A2}, we can approximate the differential operator,
$$
\mathcal{G}_\ell f({\textbf{x}_{0,k}})\approx (\mathcal{G}_\ell \mathcal{I}_{\mathbb{P}}\textbf{f}_{{\textbf{x}_0}})({\textbf{x}_{0,k}})=\sum_{|\alpha|\leq l}b_{\alpha}\mathcal{G}_\ell p_{{\textbf{x}_0},\alpha}({\textbf{x}_{0,k}}),\quad \forall k=1,...,K, \ \ell=1,...,n.
$$
For each $\ell$, the above relation can also be written in matrix form as,
\BEA
  \left[\begin{array}{c}\mathcal{G}_\ell f({\mathbf{x}_{0,1}}) \\ \vdots \\ \mathcal{G}_\ell f({\mathbf{x}_{0,K}})\end{array}\right]
 \approx
 %\left[\begin{array}{c}\mathcal{G}_\ell \mathcal{I}_{\mathbb{P}} f({\mathbf{x}_{0,1}}) \\ \vdots \\ \mathcal{G}_\ell \mathcal{I}_{\mathbb{P}} f({\mathbf{x}_{0,K}})\end{array}\right]
%=
\underbrace{\left[\begin{array}{ccc}\mathcal{G}_\ell p_{\mathbf{x}_0,\alpha(1)}(\mathbf{x}_{0,1}) & \cdots & \mathcal{G}_\ell p_{\mathbf{x}_0,\alpha(m)}(\mathbf{x}_{0,1}) \\ \vdots & \ddots & \vdots \\ \mathcal{G}_\ell p_{\mathbf{x}_0,\alpha(1)}(\mathbf{x}_{0,K}) & \cdots  & \mathcal{G}_\ell p_{\mathbf{x}_0,\alpha(m)}(\mathbf{x}_{0,K}) \end{array}\right] }_{\mathbf{B}_\ell}
\underbrace{\left[\begin{array}{c}b_{\alpha(1)} \\ \vdots \\ b_{\alpha(m)}\end{array}\right]}_{\mathbf{b}}
=\mathbf{B}_\ell(\boldsymbol{\Phi}^\top \boldsymbol{\Phi})^{-1}\boldsymbol{\Phi}^\top\mathbf{f}_{{\textbf{x}_0}},\label{eqn:G_ell} %\\
\EEA
where $\mathbf{B}_\ell$ is a $K$ by $m$ matrix with $[\mathbf{B}_\ell]_{ij}=\mathcal{G}_\ell p_{{\textbf{x}_0},\alpha(j)}({\textbf{x}_{0,i}})$ and we have used $\boldsymbol{\Phi}$ as defined in (\ref{eqn:matrix_phi}) in the last equality. Hence, the differential matrix for the operator $\mathcal{G}_\ell$ over the stencil is approximated by the $K$ by $K$ matrix,
\BEA
\mathbf{G}_\ell:=\mathbf{B}_\ell(\boldsymbol{\Phi}^\top \boldsymbol{\Phi})^{-1}\boldsymbol{\Phi}^\top.
\label{eqn:diff-op}
\EEA
Then the Laplace-Beltrami operator can be approximated at  the base point ${\mathbf{x}_0}$ as,
\BEA
\Delta_Mf ({\mathbf{x}_0})=\sum_{\ell=1}^n\mathcal{G}_\ell\mathcal{G}_\ell f ({\mathbf{x}_0})\approx \sum_{\ell=1}^n\mathcal{G}_\ell \mathcal{I}( \mathbf{G}_\ell\textbf{f}_{{\textbf{x}_0}})({\mathbf{x}_0})\approx  \big(\sum_{\ell=1}^n \mathbf{G}_\ell \mathbf{G}_\ell \textbf{f}_{{\textbf{x}_0}}\big)_1, 
\EEA
where subscript$-1$ is to denote the first element of the resulting $K$-dimensional vector. Denoting the elements of the first row of the $K$ by $K$ matrix $ \sum_{\ell=1}^n \mathbf{G}_\ell \mathbf{G}_\ell$ by $\{w_k\}_{k=1}^K$, we obtain a finite difference type approximation, that is,
\BEA
\Delta_Mf ({\mathbf{x}_0}) \approx  \sum_{k=1}^Kw_kf({\textbf{x}_{0,k}}).
\EEA
Arranging the weights at each point into each row of a sparse $N$ by $N$ matrix $\Delta^{GMLS}$, we can approximate the operator over all points by $\Delta^{GMLS}\mathbf{f}$. This GMLS Laplacian can be stabilized by employing some optimization procedures \cite%
{suchde2019meshfree,gross2020meshfree,jiang2024generalized}. In particular, one can use the following linear optimization problem \cite{jiang2024generalized}:
\BEA
\min\limits_{C,\hat{w}_1,...,\hat{w}_k}C
\EEA
with constraints
\BEA
\begin{cases}
\sum_{k=1}^K\hat{w}_kp_{\textbf{x}_0,\alpha}(\textbf{x}_{0,k})= \sum_{k=1}^Kw_kp_{\textbf{x}_0,\alpha}(\textbf{x}_{0,k}),\ \ |\alpha|\leq l,\\
\hat{w}_1<0,\\
\hat{w}_k+C\geq 0,k=2,...,K,\\
0\leq C\leq \big\vert\min\limits_{k=2,...,K}w_k\big\vert,
\end{cases}
\EEA
where the last constraint is added to guarantee an existence of the solution when $w_1<0$. In this paper, we use the same notation $\Delta^{GMLS}$ to denote the stabilized discrete operator consisting of the weights $\{\hat{w}_k\}_{k=1}^K.$
Consequently,  $\mathcal{L}^\kappa =-\text{div}_g(\kappa\textup{grad}_g \,)$ could be approximated by 
$$
\mathbf{L}^{GMLS}=-\sum_{\ell=1}^n(\mathbf{G}_\ell\kappa)\mathbf{G}_\ell-\kappa \Delta^{GMLS}.
$$

%\bibliographystyle{abbrv}
%\bibliography{arxiv_kme}

\begin{thebibliography}{10}

\bibitem{bh:16vb}
T.~Berry and J.~Harlim.
\newblock Variable bandwidth diffusion kernels.
\newblock {\em Appl. Comput. Harmon. Anal.}, 40:68--96, 2016.

\bibitem{berry2017density}
T.~Berry and T.~Sauer.
\newblock Density estimation on manifolds with boundary.
\newblock {\em Comput. Statist. Data Anal.}, 107:1--17, 2017.

\bibitem{bertalmio2001navier}
M.~Bertalmio, A.~L. Bertozzi, and G.~Sapiro.
\newblock Navier-stokes, fluid dynamics, and image and video inpainting.
\newblock In {\em Proceedings of the 2001 IEEE Computer Society Conference on
  Computer Vision and Pattern Recognition. CVPR 2001}, volume~1, pages I--I.
  IEEE, 2001.

\bibitem{bertalmio2001variational}
M.~Bertalm{\i}o, L.-T. Cheng, S.~Osher, and G.~Sapiro.
\newblock Variational problems and partial differential equations on implicit
  surfaces.
\newblock {\em Journal of Computational Physics}, 174(2):759--780, 2001.

\bibitem{cai2021deepm}
S.~Cai, Z.~Wang, L.~Lu, T.~A. Zaki, and G.~E. Karniadakis.
\newblock Deepm\&mnet: Inferring the electroconvection multiphysics fields
  based on operator approximation by neural networks.
\newblock {\em Journal of Computational Physics}, 436:110296, 2021.

\bibitem{cao2023residual}
L.~Cao, T.~O'Leary-Roseberry, P.~K. Jha, J.~T. Oden, and O.~Ghattas.
\newblock Residual-based error correction for neural operator accelerated
  infinite-dimensional bayesian inverse problems.
\newblock {\em Journal of Computational Physics}, 486:112104, 2023.

\bibitem{chen1995universal}
T.~Chen and H.~Chen.
\newblock Universal approximation to nonlinear operators by neural networks
  with arbitrary activation functions and its application to dynamical systems.
\newblock {\em IEEE Transactions on Neural Networks}, 6(4):911--917, 1995.

\bibitem{chen2020physics}
Y.~Chen, L.~Lu, G.~E. Karniadakis, and L.~Dal~Negro.
\newblock Physics-informed neural networks for inverse problems in nano-optics
  and metamaterials.
\newblock {\em Optics Express}, 28(8):11618--11633, 2020.

\bibitem{coifman2006diffusion}
R.~R. Coifman and S.~Lafon.
\newblock Diffusion maps.
\newblock {\em Applied and computational harmonic analysis}, 21(1):5--30, 2006.

\bibitem{coifman2008TuningEpsilon}
R.~R. Coifman, Y.~Shkolnisky, F.~J. Sigworth, and A.~Singer.
\newblock Graph laplacian tomography from unknown random projections.
\newblock {\em Image Processing, IEEE Transactions on}, 17(10):1891--1899,
  2008.

\bibitem{daneker2024transfer}
M.~Daneker, S.~Cai, Y.~Qian, E.~Myzelev, A.~Kumbhat, H.~Li, and L.~Lu.
\newblock Transfer learning on physics-informed neural networks for tracking
  the hemodynamics in the evolving false lumen of dissected aorta.
\newblock {\em Nexus}, 2024.

\bibitem{daneker2023systems}
M.~Daneker, Z.~Zhang, G.~E. Karniadakis, and L.~Lu.
\newblock Systems biology: Identifiability analysis and parameter
  identification via systems-biology-informed neural networks.
\newblock In {\em Computational Modeling of Signaling Networks}, pages 87--105.
  Springer, 2023.

\bibitem{deng2022approximation}
B.~Deng, Y.~Shin, L.~Lu, Z.~Zhang, and G.~E. Karniadakis.
\newblock Approximation rates of deeponets for learning operators arising from
  advection--diffusion equations.
\newblock {\em Neural Networks}, 153:411--426, 2022.

\bibitem{di2023neural}
P.~C. Di~Leoni, L.~Lu, C.~Meneveau, G.~E. Karniadakis, and T.~A. Zaki.
\newblock Neural operator prediction of linear instability waves in high-speed
  boundary layers.
\newblock {\em Journal of Computational Physics}, 474:111793, 2023.

\bibitem{dziuk2007surface}
G.~Dziuk and C.~M. Elliott.
\newblock Surface finite elements for parabolic equations.
\newblock {\em Journal of Computational Mathematics}, pages 385--407, 2007.

\bibitem{dziuk2013finite}
G.~Dziuk and C.~M. Elliott.
\newblock Finite element methods for surface pdes.
\newblock {\em Acta Numerica}, 22:289--396, 2013.

\bibitem{fan2023deep}
B.~Fan, E.~Qiao, A.~Jiao, Z.~Gu, W.~Li, and L.~Lu.
\newblock Deep learning for solving and estimating dynamic macro-finance
  models.
\newblock {\em arXiv preprint arXiv:2305.09783}, 2023.

\bibitem{fuselier2012scattered}
E.~Fuselier and G.~B. Wright.
\newblock Scattered data interpolation on embedded submanifolds with restricted
  positive definite kernels: Sobolev error estimates.
\newblock {\em SIAM Journal on Numerical Analysis}, 50(3):1753--1776, 2012.

\bibitem{fuselier2013high}
E.~J. Fuselier and G.~B. Wright.
\newblock A high-order kernel method for diffusion and reaction-diffusion
  equations on surfaces.
\newblock {\em Journal of Scientific Computing}, 56(3):535--565, 2013.

\bibitem{gh2019}
F.~Gilani and J.~Harlim.
\newblock Approximating solutions of linear elliptic pde's on a smooth manifold
  using local kernel.
\newblock {\em Journal of Computational Physics}, 395:563 -- 582, 2019.

\bibitem{greer2006improvement}
J.~B. Greer.
\newblock An improvement of a recent eulerian method for solving pdes on
  general geometries.
\newblock {\em Journal of Scientific Computing}, 29(3):321--352, 2006.

\bibitem{gross2020meshfree}
B.~J. Gross, N.~Trask, P.~Kuberry, and P.~J. Atzberger.
\newblock Meshfree methods on manifolds for hydrodynamic flows on curved
  surfaces: A generalized moving least-squares (gmls) approach.
\newblock {\em Journal of Computational Physics}, 409:109340, 2020.

\bibitem{hao2023pinnacle}
Z.~Hao, J.~Yao, C.~Su, H.~Su, Z.~Wang, F.~Lu, Z.~Xia, Y.~Zhang, S.~Liu, L.~Lu,
  et~al.
\newblock Pinnacle: A comprehensive benchmark of physics-informed neural
  networks for solving pdes.
\newblock {\em arXiv preprint arXiv:2306.08827}, 2023.

\bibitem{harlim2022graph}
J.~Harlim, S.~W. Jiang, H.~Kim, and D.~Sanz-Alonso.
\newblock Graph-based prior and forward models for inverse problems on
  manifolds with boundaries.
\newblock {\em Inverse Problems}, 38(3):035006, 2022.

\bibitem{harlim2023radial}
J.~Harlim, S.~W. Jiang, and J.~W. Peoples.
\newblock Radial basis approximation of tensor fields on manifolds: from
  operator estimation to manifold learning.
\newblock {\em Journal of Machine Learning Research}, 24(345):1--85, 2023.

\bibitem{harlim2020kernel}
J.~Harlim, D.~Sanz-Alonso, and R.~Yang.
\newblock Kernel methods for bayesian elliptic inverse problems on manifolds.
\newblock {\em SIAM/ASA Journal on Uncertainty Quantification},
  8(4):1414--1445, 2020.

\bibitem{hendrycks2016gaussian}
D.~Hendrycks and K.~Gimpel.
\newblock Gaussian error linear units (gelus).
\newblock {\em arXiv preprint arXiv:1606.08415}, 2016.

\bibitem{jiang2020ghost}
S.~W. Jiang and J.~Harlim.
\newblock Ghost point diffusion maps for solving elliptic pdes on manifolds
  with classical boundary conditions.
\newblock {\em Comm. Pure Appl. Math.}, 76(2):337--405, 2023.

\bibitem{jiang2024generalized}
S.~W. Jiang, R.~Li, Q.~Yan, and J.~Harlim.
\newblock Generalized finite difference method on unknown manifolds.
\newblock {\em Journal of Computational Physics}, 502:112812, 2024.

\bibitem{jiang2023fourier}
Z.~Jiang, M.~Zhu, D.~Li, Q.~Li, Y.~O. Yuan, and L.~Lu.
\newblock Fourier-mionet: Fourier-enhanced multiple-input neural operators for
  multiphase modeling of geological carbon sequestration.
\newblock {\em arXiv preprint arXiv:2303.04778}, 2023.

\bibitem{jiao2021one}
A.~Jiao, H.~He, R.~Ranade, J.~Pathak, and L.~Lu.
\newblock One-shot learning for solution operators of partial differential
  equations.
\newblock {\em arXiv preprint arXiv:2104.05512}, 2021.

\bibitem{jin2022mionet}
P.~Jin, S.~Meng, and L.~Lu.
\newblock Mionet: Learning multiple-input operators via tensor product.
\newblock {\em SIAM Journal on Scientific Computing}, 44(6):A3490--A3514, 2022.

\bibitem{karniadakis2021physics}
G.~E. Karniadakis, I.~G. Kevrekidis, L.~Lu, P.~Perdikaris, S.~Wang, and
  L.~Yang.
\newblock Physics-informed machine learning.
\newblock {\em Nature Reviews Physics}, 3(6):422--440, 2021.

\bibitem{lanthaler2022error}
S.~Lanthaler, S.~Mishra, and G.~E. Karniadakis.
\newblock Error estimates for deeponets: A deep learning framework in infinite
  dimensions.
\newblock {\em Transactions of Mathematics and Its Applications}, 6(1):tnac001,
  2022.

\bibitem{lee2018introduction}
J.~M. Lee.
\newblock {\em Introduction to Riemannian manifolds}.
\newblock Springer, 2018.

\bibitem{lehto2017radial}
E.~Lehto, V.~Shankar, and G.~B. Wright.
\newblock A radial basis function (rbf) compact finite difference (fd) scheme
  for reaction-diffusion equations on surfaces.
\newblock {\em SIAM Journal on Scientific Computing}, 39(5):A2129--A2151, 2017.

\bibitem{li2020fourier}
Z.~Li, N.~Kovachki, K.~Azizzadenesheli, B.~Liu, K.~Bhattacharya, A.~Stuart, and
  A.~Anandkumar.
\newblock Fourier neural operator for parametric partial differential
  equations.
\newblock {\em arXiv preprint arXiv:2010.08895}, 2020.

\bibitem{li2016convergent}
Z.~Li and Z.~Shi.
\newblock A convergent point integral method for isotropic elliptic equations
  on a point cloud.
\newblock {\em Multiscale Modeling \& Simulation}, 14(2):874--905, 2016.

\bibitem{li2017point}
Z.~Li, Z.~Shi, and J.~Sun.
\newblock Point integral method for solving poisson-type equations on manifolds
  from point clouds with convergence guarantees.
\newblock {\em Communications in Computational Physics}, 22(1):228--258, 2017.

\bibitem{liang2013solving}
J.~Liang and H.~Zhao.
\newblock Solving partial differential equations on point clouds.
\newblock {\em SIAM Journal on Scientific Computing}, 35(3):A1461--A1486, 2013.

\bibitem{liang2024solving}
S.~Liang, S.~W. Jiang, J.~Harlim, and H.~Yang.
\newblock Solving pdes on unknown manifolds with machine learning.
\newblock {\em Applied and Computational Harmonic Analysis}, page 101652, 2024.

\bibitem{lin2021operator}
C.~Lin, Z.~Li, L.~Lu, S.~Cai, M.~Maxey, and G.~E. Karniadakis.
\newblock Operator learning for predicting multiscale bubble growth dynamics.
\newblock {\em The Journal of Chemical Physics}, 154(10), 2021.

\bibitem{liu2024multi}
X.-Y. Liu, M.~Zhu, L.~Lu, H.~Sun, and J.-X. Wang.
\newblock Multi-resolution partial differential equations preserved learning
  framework for spatiotemporal dynamics.
\newblock {\em Communications Physics}, 7(1):31, 2024.

\bibitem{lu2021learning}
L.~Lu, P.~Jin, G.~Pang, Z.~Zhang, and G.~E. Karniadakis.
\newblock Learning nonlinear operators via deeponet based on the universal
  approximation theorem of operators.
\newblock {\em Nature machine intelligence}, 3(3):218--229, 2021.

\bibitem{lu2022comprehensive}
L.~Lu, X.~Meng, S.~Cai, Z.~Mao, S.~Goswami, Z.~Zhang, and G.~E. Karniadakis.
\newblock A comprehensive and fair comparison of two neural operators (with
  practical extensions) based on fair data.
\newblock {\em Computer Methods in Applied Mechanics and Engineering},
  393:114778, 2022.

\bibitem{lu2021deepxde}
L.~Lu, X.~Meng, Z.~Mao, and G.~E. Karniadakis.
\newblock {DeepXDE}: A deep learning library for solving differential
  equations.
\newblock {\em SIAM Review}, 63(1):208--228, 2021.

\bibitem{lu2022multifidelity}
L.~Lu, R.~Pestourie, S.~G. Johnson, and G.~Romano.
\newblock Multifidelity deep neural operators for efficient learning of partial
  differential equations with application to fast inverse design of nanoscale
  heat transport.
\newblock {\em Physical Review Research}, 4(2):023210, 2022.

\bibitem{lu2021physics}
L.~Lu, R.~Pestourie, W.~Yao, Z.~Wang, F.~Verdugo, and S.~G. Johnson.
\newblock Physics-informed neural networks with hard constraints for inverse
  design.
\newblock {\em SIAM Journal on Scientific Computing}, 43(6):B1105--B1132, 2021.

\bibitem{macdonald2010implicit}
C.~B. Macdonald and S.~J. Ruuth.
\newblock The implicit closest point method for the numerical solution of
  partial differential equations on surfaces.
\newblock {\em SIAM Journal on Scientific Computing}, 31(6):4330--4350, 2010.

\bibitem{mao2023ppdonet}
S.~Mao, R.~Dong, L.~Lu, K.~M. Yi, S.~Wang, and P.~Perdikaris.
\newblock Ppdonet: Deep operator networks for fast prediction of steady-state
  solutions in disk--planet systems.
\newblock {\em The Astrophysical Journal Letters}, 950(2):L12, 2023.

\bibitem{mao2021deepm}
Z.~Mao, L.~Lu, O.~Marxen, T.~A. Zaki, and G.~E. Karniadakis.
\newblock Deepm\&mnet for hypersonics: Predicting the coupled flow and
  finite-rate chemistry behind a normal shock using neural-network
  approximation of operators.
\newblock {\em Journal of computational physics}, 447:110698, 2021.

\bibitem{marcati2023exponential}
C.~Marcati and C.~Schwab.
\newblock Exponential convergence of deep operator networks for elliptic
  partial differential equations.
\newblock {\em SIAM Journal on Numerical Analysis}, 61(3):1513--1545, 2023.

\bibitem{memoli2004implicit}
F.~M{\'e}moli, G.~Sapiro, and P.~Thompson.
\newblock Implicit brain imaging.
\newblock {\em NeuroImage}, 23:S179--S188, 2004.

\bibitem{moya2024conformalized}
C.~Moya, A.~Mollaali, Z.~Zhang, L.~Lu, and G.~Lin.
\newblock Conformalized-deeponet: A distribution-free framework for uncertainty
  quantification in deep operator networks.
\newblock {\em arXiv preprint arXiv:2402.15406}, 2024.

\bibitem{nair2010rectified}
V.~Nair and G.~E. Hinton.
\newblock Rectified linear units improve restricted {Boltzmann} machines.
\newblock In {\em ICML}, 2010.

\bibitem{nitschke2020liquid}
I.~Nitschke, S.~Reuther, and A.~Voigt.
\newblock Liquid crystals on deformable surfaces.
\newblock {\em Proceedings of the Royal Society A}, 476(2241):20200313, 2020.

\bibitem{pang2019fpinns}
G.~Pang, L.~Lu, and G.~E. Karniadakis.
\newblock fpinns: Fractional physics-informed neural networks.
\newblock {\em SIAM Journal on Scientific Computing}, 41(4):A2603--A2626, 2019.

\bibitem{petras2018rbf}
A.~Petras, L.~Ling, and S.~J. Ruuth.
\newblock An rbf-fd closest point method for solving pdes on surfaces.
\newblock {\em Journal of Computational Physics}, 370:43--57, 2018.

\bibitem{piret2012orthogonal}
C.~Piret.
\newblock The orthogonal gradients method: A radial basis functions method for
  solving partial differential equations on arbitrary surfaces.
\newblock {\em Journal of Computational Physics}, 231(14):4662--4675, 2012.

\bibitem{raissi2019pinn}
M.~Raissi, P.~Perdikaris, and G.~Karniadakis.
\newblock Physics-informed neural networks: A deep learning framework for
  solving forward and inverse problems involving nonlinear partial differential
  equations.
\newblock {\em Journal of Computational Physics}, 378:686--707, 2019.

\bibitem{rauter2018finite}
M.~Rauter and {\v{Z}}.~Tukovi{\'c}.
\newblock A finite area scheme for shallow granular flows on three-dimensional
  surfaces.
\newblock {\em Computers \& Fluids}, 166:184--199, 2018.

\bibitem{ruuth2008simple}
S.~J. Ruuth and B.~Merriman.
\newblock A simple embedding method for solving partial differential equations
  on surfaces.
\newblock {\em Journal of Computational Physics}, 227(3):1943--1961, 2008.

\bibitem{schoenborn1999kinetics}
O.~Schoenborn and R.~C. Desai.
\newblock Kinetics of phase ordering on curved surfaces.
\newblock {\em Journal of statistical physics}, 95(5):949--979, 1999.

\bibitem{shankar2015radial}
V.~Shankar, G.~B. Wright, R.~M. Kirby, and A.~L. Fogelson.
\newblock A radial basis function (rbf)-finite difference (fd) method for
  diffusion and reaction--diffusion equations on surfaces.
\newblock {\em Journal of scientific computing}, 63(3):745--768, 2015.

\bibitem{shi2017weighted}
Z.~Shi, S.~Osher, and W.~Zhu.
\newblock Weighted nonlocal laplacian on interpolation from sparse data.
\newblock {\em Journal of Scientific Computing}, 73(2):1164--1177, 2017.

\bibitem{suchde2019meshfree}
P.~Suchde and J.~Kuhnert.
\newblock A meshfree generalized finite difference method for surface pdes.
\newblock {\em Computers \& Mathematics with Applications}, 78(8):2789--2805,
  2019.

\bibitem{tian2009segmentation}
L.~Tian, C.~B. Macdonald, and S.~J. Ruuth.
\newblock Segmentation on surfaces with the closest point method.
\newblock In {\em 2009 16th IEEE International Conference on Image Processing
  (ICIP)}, pages 3009--3012. IEEE, 2009.

\bibitem{tripura2023wavelet}
T.~Tripura and S.~Chakraborty.
\newblock Wavelet neural operator for solving parametric partial differential
  equations in computational mechanics problems.
\newblock {\em Computer Methods in Applied Mechanics and Engineering},
  404:115783, 2023.

\bibitem{wang2023learning}
H.~Wang, L.~Lu, S.~Song, and G.~Huang.
\newblock Learning specialized activation functions for physics-informed neural
  networks.
\newblock {\em arXiv preprint arXiv:2308.04073}, 2023.

\bibitem{wang2021learning}
S.~Wang, H.~Wang, and P.~Perdikaris.
\newblock Learning the solution operator of parametric partial differential
  equations with physics-informed deeponets.
\newblock {\em Science advances}, 7(40):eabi8605, 2021.

\bibitem{Wendland2005Scat}
H.~Wendland.
\newblock {\em Scattered Data Approximation}.
\newblock Cambridge University Press, 2005.

\bibitem{wu2023effective}
W.~Wu, M.~Daneker, M.~A. Jolley, K.~T. Turner, and L.~Lu.
\newblock Effective data sampling strategies and boundary condition constraints
  of physics-informed neural networks for identifying material properties in
  solid mechanics.
\newblock {\em Applied mathematics and mechanics}, 44(7):1039--1068, 2023.

\bibitem{wu2024identifying}
W.~Wu, M.~Daneker, K.~T. Turner, M.~A. Jolley, and L.~Lu.
\newblock Identifying heterogeneous micromechanical properties of biological
  tissues via physics-informed neural networks.
\newblock {\em arXiv preprint arXiv:2402.10741}, 2024.

\bibitem{xu2003eulerian}
J.-J. Xu and H.-K. Zhao.
\newblock An eulerian formulation for solving partial differential equations
  along a moving interface.
\newblock {\em Journal of Scientific Computing}, 19(1):573--594, 2003.

\bibitem{yan2022ghost}
Q.~Yan, S.~Jiang, and J.~Harlim.
\newblock Kernel-based methods for solving time-dependent advection-diffusion
  equations on manifolds.
\newblock {\em Journal of Scientific Computing}, 94(1), 2023.

\bibitem{yan2023spectral}
Q.~Yan, S.~W. Jiang, and J.~Harlim.
\newblock Spectral methods for solving elliptic pdes on unknown manifolds.
\newblock {\em Journal of Computational Physics}, 486:112132, 2023.

\bibitem{yazdani2020systems}
A.~Yazdani, L.~Lu, M.~Raissi, and G.~E. Karniadakis.
\newblock Systems biology informed deep learning for inferring parameters and
  hidden dynamics.
\newblock {\em PLoS computational biology}, 16(11):e1007575, 2020.

\bibitem{yin2024dimon}
M.~Yin, N.~Charon, R.~Brody, L.~Lu, N.~Trayanova, and M.~Maggioni.
\newblock Dimon: Learning solution operators of partial differential equations
  on a diffeomorphic family of domains.
\newblock {\em arXiv preprint arXiv:2402.07250}, 2024.

\bibitem{yu2022gradient}
J.~Yu, L.~Lu, X.~Meng, and G.~E. Karniadakis.
\newblock Gradient-enhanced physics-informed neural networks for forward and
  inverse pde problems.
\newblock {\em Computer Methods in Applied Mechanics and Engineering},
  393:114823, 2022.

\bibitem{zhang2019quantifying}
D.~Zhang, L.~Lu, L.~Guo, and G.~E. Karniadakis.
\newblock Quantifying total uncertainty in physics-informed neural networks for
  solving forward and inverse stochastic problems.
\newblock {\em Journal of Computational Physics}, 397:108850, 2019.

\bibitem{zhang2024d2no}
Z.~Zhang, C.~Moya, L.~Lu, G.~Lin, and H.~Schaeffer.
\newblock D2no: Efficient handling of heterogeneous input function spaces with
  distributed deep neural operators.
\newblock {\em Computer Methods in Applied Mechanics and Engineering},
  428:117084, 2024.

\bibitem{zhao2001fast}
H.-K. Zhao, S.~Osher, and R.~Fedkiw.
\newblock Fast surface reconstruction using the level set method.
\newblock In {\em Proceedings IEEE Workshop on Variational and Level Set
  Methods in Computer Vision}, pages 194--201. IEEE, 2001.

\bibitem{zhu2023fourier}
M.~Zhu, S.~Feng, Y.~Lin, and L.~Lu.
\newblock Fourier-deeponet: Fourier-enhanced deep operator networks for full
  waveform inversion with improved accuracy, generalizability, and robustness.
\newblock {\em Computer Methods in Applied Mechanics and Engineering},
  416:116300, 2023.

\bibitem{zhu2023reliable}
M.~Zhu, H.~Zhang, A.~Jiao, G.~E. Karniadakis, and L.~Lu.
\newblock Reliable extrapolation of deep neural operators informed by physics
  or sparse observations.
\newblock {\em Computer Methods in Applied Mechanics and Engineering},
  412:116064, 2023.

\end{thebibliography}

\end{document}